\documentclass[10pt]{article}

\usepackage[utf8]{inputenc}

\usepackage{hyperref}
\usepackage{verbatim} %to allow long "comment-out-'s"

\usepackage{graphicx}
\usepackage{caption}
\usepackage{subcaption}
\usepackage{amssymb}
\usepackage{ textcomp } %Promille sign
\usepackage{color}
\usepackage{enumerate}
\usepackage{mathrsfs}
\usepackage{multicol}

\usepackage[ruled,vlined, onelanguage]{algorithm2e}

\usepackage{amsopn}

\usepackage{amsmath}
\usepackage{amsthm}
\usepackage{amssymb}

\usepackage{multicol}

\usepackage[ruled,vlined]{algorithm2e}

\newcommand{\C}{\mathbb{C}}

\newcommand{\K}{\mathbb{K}}
\newcommand{\N}{\mathbb{N}}

\newcommand{\R}{\mathbb{R}}

\newcommand{\calA}{\mathcal{A}}

\newcommand{\calC}{\mathcal{C}}
\newcommand{\calD}{\mathcal{D}}
\newcommand{\calE}{\mathcal{E}}
\newcommand{\calF}{\mathcal{F}}
\newcommand{\calH}{\mathcal{H}}
\newcommand{\calI}{\mathcal{I}}

\newcommand{\calL}{\mathcal{L}}
\newcommand{\calM}{\mathcal{M}}
\newcommand{\calP}{\mathcal{P}}

\newcommand{\calS}{\mathcal{S}}

\newcommand{\vphi}{\varphi}

\newcommand{\abs}[1]{\vert #1 \vert}
\newcommand{\norm}[1]{\Vert #1 \Vert}
\newcommand{\abss}[1]{\left\vert #1 \right\vert}
\newcommand{\normm}[1]{\left\Vert #1 \right\Vert}
\newcommand{\set}[1]{\left\lbrace #1\right\rbrace}
\newcommand{\sse}{\subseteq}
\newcommand{\sprod}[1]{\left\langle #1 \right\rangle}

\newcommand{\sph}{\mathbb{S}}

\newcommand{\prb}[1]{\mathbb{P}\left( #1 \right)}
\newcommand{\erw}[1]{\mathbb{E}\left( #1 \right)}

\usepackage{dsfont}
\newcommand{\ind}{\mathds{1}}

\newcommand{\geqsim}{\gtrsim}
\newcommand{\leqsim}{\lesssim}

\newcommand{\hatsigma}{\widehat{\sigma}}
\newcommand{\hatphi}{\widehat{\phi}}

\DeclareMathOperator{\clonv}{\overline{conv}}
\DeclareMathOperator{\cone}{cone}

\DeclareMathOperator{\qri}{qri}

\DeclareMathOperator{\dist}{dist}
\DeclareMathOperator{\spn}{span}

\DeclareMathOperator{\supp}{supp}
\DeclareMathOperator{\id}{id}
\DeclareMathOperator{\argmin}{argmin}

\DeclareMathOperator{\ran}{ran}

\DeclareMathOperator{\rank}{rank}

\DeclareMathOperator{\tr}{tr}

\DeclareMathOperator{\dom}{dom}

\DeclareMathOperator{\ri}{ri}

\DeclareMathOperator{\re}{Re}
\DeclareMathOperator{\im}{Im}

\newcommand{\st}{\text{ subject to }}

\newtheorem{lem}{Lemma}
\newtheorem{prop}[lem]{Proposition}

\newtheorem{theo}[lem]{Theorem}
\newtheorem{cor}[lem]{Corollary}

\newtheorem{defi}[lem]{Definition}
\newtheorem{rem}[lem]{Remark}

\theoremstyle{definition}
\newtheorem{example}[lem]{Example}

\numberwithin{lem}{section}

\usepackage{etoolbox}
\let\bbordermatrix\bordermatrix
\patchcmd{\bbordermatrix}{8.75}{4.75}{}{}
\patchcmd{\bbordermatrix}{\left(}{\left[}{}{}
\patchcmd{\bbordermatrix}{\right)}{\right]}{}{}

\title{Soft Recovery With General Atomic Norms}

\usepackage[affil-it]{authblk}

\author{Axel Flinth\thanks{E-mail: {\tt flinth@math.tu-berlin.de}}}
\affil{Institut für Mathematik \\ Technische Universität Berlin}

\begin{document}

\maketitle

\begin{abstract}

	This paper describes a dual certificate condition on a linear measurement operator $A$ (defined on a Hilbert space $\calH$ and having finite-dimensional range) which guarantees that an \emph{atomic norm minimization}, in a certain sense, will be able to \emph{approximately} recover a structured signal $v_0 \in \calH$ from measurements $Av_0$. Put very streamlined, the condition implies that peaks in a sparse decomposition of $v_0$ are close the the support of the atomic decomposition of the solution $v^*$. The condition applies in a relatively general context - in particular, the space $\calH$ can be infinite-dimensional. The abstract framework isapplied to several concrete examples, one example being super-resolution. In this process, several novel results which are interesting on its own are obtained.
	
\end{abstract}

\section{Introduction}

Put very generally, the field of \emph{compressed sensing} \cite{candes2006robust} deals with (linear) inverse problems $Ax=b$ which are a priori ill-posed, but given that there exists a solution with a certain structure, the inverse problem becomes solvable. The most well-known of these structures is probably \emph{sparsity} of the ground truth signals, but many other structures are possible -- low rank for matrices being a prominent non-trivial example.  

Many structured of this kind can be modeled with the help of the so called \emph{synthesis formulation}: Given a collection of vectors $\Phi=(\vphi_x)_{x \in I}$ -- or \emph{dictionary}\footnote{In this paper, we will adopt an understanding of what a dictionary is which is a bit more restrictive than just being a collection of vectors. For the introduction, these subtleties are however not important.} -- we think of a signal $v_0$ to be structured if it possesses a \emph{sparse atomic decomposition} in the dictionary:
\begin{align*}
	v_0= \sum_{x \in I_0} c_x^0 \vphi_{x},
\end{align*}
with $\abs{I_0} \ll \abs{I}$ and scalars $c_x^0$. A general strategy for recovering such signals from linear measurements $b=Av_0$ was described in \cite{chandrasekaran2012convex}, namely \emph{atomic norm minimization}:
\begin{align*}
	 \min_v \norm{v}_\calA \st Av=b. \tag{$\calP_\calA$}
\end{align*}
The atomic norm is thereby defined as $\inf \set{ t>0 \ \vert \ v \in t \clonv \Phi}$. In the case of a finite dictionary, it is not hard to convince oneself that by taking $v = \sum_{i \in I} c_i \vphi_i$, one alternatively can solve the program
\begin{align*}
	\min_c \norm{c}_1 \st A\left(\sum_{i \in I} c_i \vphi_i \right) = b. \tag{$\calP_1$}
\end{align*}

 A generic compressed sensing result about the program $(\calP_\calA)$ would read something like this:
\begin{quote}
	Assume that the measurement process $A$ fulfills a certain set of properties and the ground truth signal $v_0$ has a sparse atomic decomposition in $(\vphi_x)_{x\in I}$. Then the solution of $(\calP_\calA)$ is exactly equal to $v_0$.
\end{quote} 
In the following, we will call such guarantees \emph{exact recovery guarantees}. To the best knowledge of the author, there is no theorem that holds in this generality in the literature. There is however a vast body of literature concerning a zoo of important special examples. Some examples of such dictionaries and corresponding atomic norms include the standard dictionary $(e_i)_{i=1}^n$ and the $\ell^1$-norm \cite{CandesRombergTao2006}, the dictionary of rank-one operators $(uv^*)_{u \in \sph^{k-1}, v \in \sph^{n-1}}$ and the nuclear norm \cite{gross2011golfing}, and the dictionary of matrices $(ue_i^*)_{u \in \sph^{k-1}, i=1, \dots n}$ and the $\ell_{12}$-norm \cite{Eldar2009BlockSparse}. More examples can be found in \cite{chandrasekaran2012convex}. \newline

For \emph{coherent} dictionaries (i.e. dictionaries for which $\sup_{x \neq x'} \tfrac{\abs{\sprod{\vphi_x, \vphi_{x'}}}}{\norm{\vphi_x}\norm{\vphi_{x'}}}$ is close to one), the recovery of a structured signal can be very unstable. To explain why this is the case,  notice that the high coherence will cause the map 
\begin{align*}
	D: c \to \sum_{i \in I} c_i \vphi_i
\end{align*}
to be highly ill-conditioned, i.e., there will exist relatively sparse $c$ with large $\ell_1$-norm with $\norm{Dc}\approx 0$. Hence, it could happen that for $v, v'$ relatively close to each other, the corresponding solutions $c, c'$ to $(\calP_1)$ are radically different.  The ill-conditioning however goes both ways: Although $c^0$ and $c^*$ are far away from each other, $Dc^0$ can still be close to $Dc^*$. In applications, the recovery of the \emph{signal} $Dc^0$ is often what one is genuinely interested in.

Sometimes however, the recovery of the \emph{representation} $c^0$ itself is important. To see this, let us imagine that the signal $v_0$ is a recording of a tone played by a musical instrument, and we need to recover it from a compressed (or in some other way altered) version $Av_0$. Since only one tone is played, $v_0$ should be sparsely representable in the dictionary of complex exponentials $\vphi_\lambda = \exp(i\lambda \cdot t)$, say $v_0 = \sum_{i=1}^s c_i^0 \vphi_{\lambda_i}$ (we cannot expect $s=1$ due to overtones). For a person just listening to the tone, the recovery of $v_0$ is the important thing. For a person wanting to play the same tone on a different instrument, it is however also important to recover $c^0$. We certainly cannot hope to recover $c^0$ exactly if we cannot recover $v_0$ exactly, but an approximate recovery will in fact suffice: If an automated recovery algorithm says that the tone that was played had a frequency of $439.8$ Hz, an experienced musician will be able to identify the note as an $A$ in the fourth octave (which has a standard frequency of $440$ Hz). In order to be able to trust the recovery algorithm in the first place, we need a method of proving that it approximately recovers the correct frequency.\newline

Up until now, there has been little research providing rigororous backing up statements about the type of recovery described above. Note that a rigorous claim about this type of approximate recovery is something radically different compared to recovery from \emph{noisy measurements} and \emph{instance optimality} \cite{wojtaszczyk2010stability} (i.e. recovery of approximately structured signals). What we are interested in is instead a result of the following flavour:
\begin{quote}
Assume that the measurement process $A$ fulfills a certain set of properties and the ground truth signal $v_0$ has a sparse atomic decomposition in $(\vphi_x)_{x\in I}$. Then the solution of $(\calP_\calA)$ has an atomic decomposition which is approximately equal to the one of $v_0$.
\end{quote}
 We call this type of guarantee \emph{soft recovery}. This concept  was originally introduced in \cite{Flinth2016Soft} by the author in the special context of recovery of \emph{column sparse matrices}. In this paper, we will analyze this concept in a significantly more general setting:
\begin{itemize}
	\item First and foremost, as has already been advertized, we will prove our guarantee in an atomic norm framework. This will reveal the underlying structures of the argument in \cite{Flinth2016Soft}, which implicitly uses the fact that the $\ell_{12}$-norm is an atomic norm (more information about this can be found in Section \ref{sec:L12}). This generality will allow us to apply the results in very general contexts. In particular, the dictionaries will be allowed to be uncountable.
	\item The possibility of an uncountable dictionary behoves us to study programs on the space of Radon measures of finite total variation. This is mathematically much more subtle than the programs appearing in \cite{Flinth2016Soft}.
	\item We will allow the signals to be elements of infinite dimensional spaces. This will enable us to provide an analysis of \emph{super resolution} in Section \ref{sec:SuperResolution}.
	\item Finally, we will signal in complex, and not only real, ambient spaces. This is not as conceptually different as the above generalizations, but very handy when applying the results.
\end{itemize}

\subsection{Main Results}
Let us give a brief, mathematical description of the main result \ref{th:soft} of this paper. For simplicity, we will stay in the case of a finite  normalized dictionary  $(\vphi_i)_{i=1}^d$ -- in the case of a infinite dictionary, in particular when the parameter $x$ is continuous, we have to be a bit more careful.

Let the signal $v_0 \in \calH$, where $\calH$ is a Hilbert space over $\K$, have a sparse atomic decomposition $\sum_{i =1}^d c_i \vphi_i$, where $c \in \K^d$ is sparse.
 Let further $i_0 \in \supp c$, $\sigma \geq 1$ and $t>0$ be arbitrary. The main result of this paper states that  provided $A$ satisfies a dual certificate-type condition (depending on $\sigma$, $t$ and also $i_0$), any minimizer $v_*$ of $(\calP_\calA)$ has the following property: There exist scalars $c_i^* \in \K$ with $v= \sum_{i =1 }^d c_i^* \vphi_i$, $\norm{c^*}_1 = \norm{v_*}_\calA$  and some $i_* \in \set{1, \dots d}$ with $\abs{ \sprod{\vphi_i, \vphi_{i*}}} \geq \frac{t}{\sigma}$. Since the dictionary elements are normalized, this directly corresponds to $\vphi_{i_0}$ and $\vphi_{i_*}$ being close. Put differently, \emph{if the atomic decomposition of $x_0$ has a 'peak' at $\vphi_{i_0}$, the condition will guarantee that all atomic decompositions of $x_*$ have a support which contains an an element $\vphi_{i_*}$ close to $\vphi_{i_0}$}. See also Figure \ref{fig:softRec}.  \newline
 
 \begin{figure}
 \centering
 \includegraphics[scale=.5]{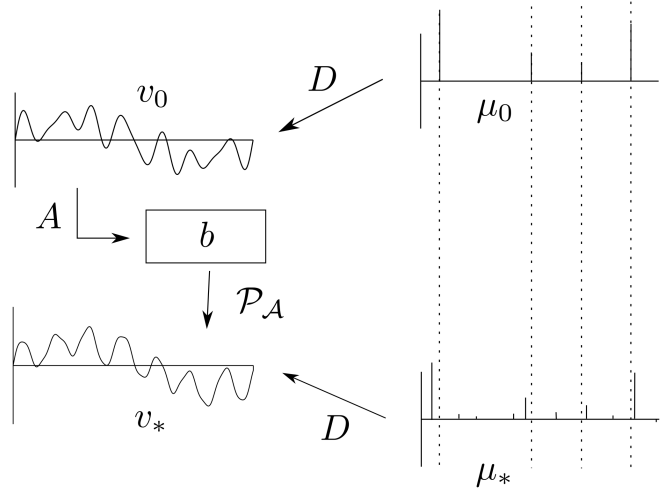}
 \caption{\label{fig:softRec} A graphical depiction of the meaning of soft recovery. The ground truth signal $v_0$ has a sparse atomic decomposition $\mu_0$. It is measured with the operator $A$ and is then tried to be recovered by solving the nuclear norm minimization program $(\calP_\calA)$. The recovered signal $v_* \neq v_0$ has an atomic decomposition $\mu_*$ which is not exactly equal to $\mu_0$, but it still has peaks at roughly the same places as $\mu_0$.}
 \end{figure}
 
 The generality of the main result \ref{th:soft} makes it applicable in many different contexts. In this paper, we will consider four such examples: 
 \begin{itemize}
 \item \emph{Nuclear Norm Minimization}, 
 \item \emph{$\ell_{12}$-minimization}, 
 \item \emph{Component Separation} and 
 \item \emph{Superresolution}.
 \end{itemize}
  In the two latter cases, the new theory will actually allow us to prove novel results that are interesting on their own. For component separation, which we will investigate in a finite dimensional setting, we will provide a guarantee for an individual index to be recovered, which for indices corresponding to large components can be much stronger than guarantees purely relying on sparsity assumptions. For superresolution, we will be able to prove soft recovery guarantees in very general settings. To the best knowledge of the author, most concrete, rigorous recovery guarantees for infinite dimensional superresolution have been given only for very particular measurement processes. 
  
  The successful application of the theory in the above four examples -- each relatively different from the others -- conceptually shows that although the main result may seem relatively abstract and complicated, it is actually is graspable enough to be applicable in a wide range of contexts, producing novel results in in the process. Without going to much into detail, the intuitive reason why this is, is that the soft recovery condition asks for the dual certificate \emph{approximate} a certain 'ideal certificate', where as exact recovery conditions typically asks the dual certificate to \emph{interpolate} certain values of that ideal certificate (the meaning of this will become particularly clear when we discuss super resolution in Section \ref{sec:SuperResolution}). Since approximating, from a functional analytic perspective, often is easier than interpolating, this also suggests that the soft recovery condition in the future could be applied in other contexts than conditions based exact dual certificate conditions can be.
  
 \subsection{Related work} 
In the paper \cite{DuvalPeyre2015}, the authors prove a recovery statement of ''soft type''  in the context of recovery of trains of $\delta$-peaks $\sum_{i=1}^s c_i \delta_{x_i}$ on $[0,2\pi]$. They consider the case of noisy measurements of convolution type $\Phi \mu= \int_{0}^{2\pi} \phi(x-t)d\mu(t) + n$, where $\phi \in \calC^2(0,2\pi)$ and $n$ is small. Theorem 2 of the referenced paper states that \emph{under an exact recovery condition}, the position of the peaks of the solutions of a regularized version of $(\calP_\calA)$ will converge to the the positions of the peak of the ground truth. A similar statement is provided in \cite{boyer2016adapting} about the so-called Beurling Lasso in the special case of the measurements being of Fourier type.  Although the (strong) statements in the mentioned papers bear very close resemblances to our main result, and the mathematics presented in them certainly is very elegant, they have one drawback, namely that the exact recovery condition still has to hold in order for them to be applied. Furthermore, it is not clear if it generalizes to more general measurement  operators. Our condition will in some cases be weaker (see in particular Section \ref{sec:Nuc}) than and/or be easier to check (see in particular Section \ref{sec:SuperResProof}) than the exact recovery condition. In particular, we expect them to hold for more general measurement operators.

 The concept of atomic norms was introduced in \cite{chandrasekaran2012convex}.  Much of the work will rely on techniques developed in e.g. \cite{candes2014towards,DuvalPeyre2015} for analyzing the use of $TV$-norm minimization for recovering spike trains. 
 
 The concrete examples which we consider in Section \ref{sec:Appl} have all been thoroughly investigated before. We will refer to related literature in respective sections.

	\subsection{Organization of the Paper} In Section \ref{sec:Atomic}, we will define the concept of atomic norm for general dictionaries in general Hilbert spaces. We will take a route which makes the development of our main result in Section \ref{sec:main} as comfortable as possible. In Section \ref{sec:Appl}, we will investigate which impact our main result has in four different compressed sensing-flavored settings. Finally, technical proofs and technical discussions not crucially important to the main body of the text can be found in Section \ref{sec:proofs}.

\section{Atomic Norms and Compressed Sensing in Infinite Dimensions} \label{sec:Atomic}

The aim of classical compressed sensing is to recover sparse vectors $x_0 \in \K^n$ from linear measurements $Ax_0$ with $A \in \K^{n,m}$, where $\K$ denotes one of the fields $\R$ or $\C$. A sparse vector could thereby be viewed as a vector having a \emph{sparse representation} in the standard basis $(e_i)_{i=1}^n$.  In the article \cite{chandrasekaran2012convex}, this perspective was widened: The authors defined so called \emph{atomic norms} for general, finite-dimensional dictionaries $(d_i)_{i\in I}$ and showed that they can be used to recover signals $x_0$ having sparse representation in said dictionary, i.e. for signals of the type
\begin{align*}
	x_0 = \sum_{i \in I_0} c_i d_i,
\end{align*}
where $I_0$ is a subset of $I$ with $\abs{I_0}$ small.

In this section, we aim to present, in an as condensed as possible manner, a definition of an atomic norm with respect to a 
dictionary $(\vphi_x)_{x \in I}$ in a Hilbert space $\calH$. The definition given here is different from the one given in \cite{chandrasekaran2012convex}, and will upon first inspection perhaps seem more complicated. We will however need this description when proving our main theorem.

\subsection{Atomic Norms}

Let us begin by defining what we will understand as the atomic norm with respect to a dictionary $(\vphi_x)_{x\in I}$ (we will in particular define what a \emph{dictionary} is).

 We will assume throughout that $I$ is a locally compact separable metric space. $\calC_0(I)$ will denote the space of continuous functions $f: I \to \K$  \emph{vanishing at infinity}, meaning that for every $\epsilon >0$, there exists a compact set $K$ with $\abs{f(x)}< \epsilon$ for $x \notin K$. This space becomes a Banach space when equipped with the norm $\norm{f}_\infty= \sup_{x \in I} \abs{f(x)}$. It is well known that the dual of $\calC_0(I)$ is given by the set of complex Radon measures with bounded total variation $\calM(I)$ on $I$. The dual pairing is given by
\begin{align*}
	\sprod{\mu, f}_{\calM(I), \calC_0(I)} = \int_I f d\mu.
\end{align*}
We will consequently leave out the subscripts of the pairing and use $\sprod{\cdot  ,  \cdot}$ interchangeably as notation for this dual pairing and the scalar product on $\calH$  - it will be clear from the context what is meant.

The total variation norm $\norm{\mu}_{TV}$ can either be defined as the norm of $\mu$ as a functional on $\calC_0(I)$, or $\abs{\mu}(I)$, where $\abs{\mu}$ is the \emph{total variation} of $\mu$:
\begin{align*}
	\abs{\mu}(A) = \sup \set{ \sum_{n=1}^\infty \abs{\mu(A_n)} \ \vert \ A= \dot{\bigcup_{n \in \N}} A_n},
\end{align*}
where the dot on top of the $\cup$ indicates that the union is disjoint. $\calM(I)$ equipped with this norm is a Banach space. We will however almost exclusively view it as a locally convex vector space, equipped with the weak-$*$ topology. The dual space of $\calM(I)$ viewed this way is isometrically isomorph to $\calC_0(I)$. The use of measures instead of sequences will become clear in the sequel: Basically put, the topology of $\calM(I)$ has nice properties. The definition was inspired by, e.g., \cite{DuvalPeyre2015}. \newline

We are now ready to define the meaning of a dictionary.

\begin{defi}\label{defi:dict} Let $I$ be a (locally compact and separable) metric space and $\calH$ a Hilbert space. We consider a collection $(\vphi_x)_{x \in I}$ of vectors in $\calH$ to be a \emph{dictionary} if the \emph{test map} 
$$T: \calH \to \calC_0(I), v \mapsto \sprod{v, \vphi_x}$$
 is well-defined and bounded.
\end{defi}

\begin{example} \label{ex:E} \begin{enumerate}
	 \item Let $I= \N$. By equipping $\N$ with the discrete metric, it becomes a locally compact, separable metric space. It is not hard to convince oneself that $\calC_0(\N)$ is the space of sequences $(c_i)$ converging to zero,, $\calM(I)$ is isomorphic $\ell^1(\N)$ (still equipped with the weak-$*$-topology!), and that the set of dictionaries coincides with the set of sequences of vectors $(\phi_n) \sse \calH$ with $\sup_{n\in \N} \norm{\phi_n} <\infty$. 
	 
	Even the even more rudimentary case that \emph{$I= \set{1, \dots, d}$ is finite}, $\calC_0(I) \sim (\K^d, \norm{ \cdot  }_\infty$ and $\calM(I) \sim (\K^d, \norm{ \cdot  }_1)$.\newline
	
	\item  Let $U$ and $V$ be separable Hilbert spaces and $\calH$ be the space of Hilbert-Schmidt operators from $U$ to $V$ i.e.
	\begin{align*}
		\set{ A \in \calL(U,V) \ \vert \ \norm{A}_{\calH}^2= \tr(A^*A) <\infty}.
	\end{align*}	
	Let $I$ be the set of normalized rank-one operators from $U$ to $V$ with norm less than $1$:
	\begin{align*}
		I = \set{v \otimes u \ \vert \ u \in U, v \in V, \norm{u}_U,\norm{v}_V \leq 1}.
	\end{align*}
	$I$ can be equipped with the metric induced by $\norm{\cdot}_\calH$, but then, $I$ is not locally compact. We can however make it so by equipping it with the weak(-$*$) topology induced by $\norm{\cdot}_\calH$. This still makes it possible to view it as a metric space, since the weak-$*$-topology on the unit ball of a dual $X^*$ of a separable Banach space $X$ is metrizable. $I$ is separable due to the separability of $U$ and $V$, and even compact due to the Banach-Alaoglu theorem.
	
	Now consider the collection of vectors defined by $\vphi_{v \otimes u}=v \otimes u$. We claim that this is a dictionary. To prove this, we only need to note that for all $A\in \calH$, $v \otimes u \to \sprod{A, v\otimes u}_\calH$ defines a continuous map vanishing at infinity (the compactness of $I$ makes the 'vanishing at infinity' part trivial), and
	\begin{align*}
		\sup_{ \norm{v \otimes u}_\calH \leq 1} \sprod{A, v\otimes u} \leq \norm{A}.
	\end{align*}
	The test map is hence bounded. \newline
	
	\item  Consider $I=\R$ with the standard metric (that is a separable, locally compact space). For a function $\phi \in \calS(\R)$\footnote{$\calS(\R)$ denotes the space of Schwartz functions on $\R$. Its topological dual is given by the space of \emph{tempered distributions} $\calS'(\R)$.} consider the Hilbert space $\calE$ defined as
	\begin{align*}
		\calE = \overline{\set{ v\in \calS'(\R):  \norm{v *\phi}_2^2 < \infty}}
	\end{align*}
	with norm $\norm{\cdot}_\calE$ induced by the scalar product
	\begin{align*}
		\sprod{v, w}_\calE = \sprod{v * \phi, w * \phi}_{L^2(\R)}.
	\end{align*}
	The closure in the definition of $\calE$ is of course with respect to the norm $\norm{\cdot}_\calE$. We claim that	$(\vphi_x)_{x \in \R} = (\delta_x)_{x \in \R}$ is a dictionary. To see this, let $v \in \calE$ be arbitrary. Then we have for all $x \in \R$
	\begin{align*}
		\sprod{v, \delta_x}_\calE = \sprod{v * \phi , \delta_x *\phi}_{L^2(\R)} = \sprod{\widehat{v} \widehat{\phi}, \widehat{\delta}_x \cdot \widehat{\phi}}_{L^2(\R)} = \int_{\R} \widehat{v}(t)\abs{\widehat{\phi}(t)}^2 \exp(-itx) dt= \widehat{g}(x),
	\end{align*}
	where we defined $a(t)=\widehat{v}(t)\abs{\widehat{\phi}(t)}^2 $. The function $a$ is as a product of the two $L^2$-functions $\widehat{v}\widehat{\phi}$ and $\overline{\widehat{\phi}}$ in $L^1(\R)$. The theorem of Lebesgue implies that $\widehat{a}$ is in $\calC_0(\R)$, i.e. that the test map is mapping into $\calC_0(\R)$.  The boundedness of the test map is given by the inequality $\abs{\sprod{v, \delta_x}}_\calE \leq \norm{v}_\calE \norm{\delta_x}_\calE$. 
	\end{enumerate}
\end{example}

For a dictionary, we can define a \emph{synthesis operator} as follows: For $v \in \calH$ and $\mu \in \calM(I)$ arbitrary, the integral 
\begin{align*}
	\int_I \sprod{v,\vphi_x} d\mu(x)
\end{align*}
is well-defined. The  map $v \to \int_I \sprod{v,\phi_x} d\mu(x)$ is furthermore bounded (its norm is smaller than the product of the norm of the test map of $(\vphi_x)$ and $\norm{\mu}_{TV}$), so it defines a functional on $\calH$, and therefore by Riesz duality also a $w$ in $\calH$. Hence, we may  define a \emph{synthesis operator} by
\begin{align*}
	D: \calM(I) \to \calH, \mu \to \int_I \vphi_x d\mu(x).
\end{align*}
The dual operator $D^* : \calH \to \calC_0(I)$ is then given by the test map of the dictionary.
With the help of the synthesis operator, we define the atomic norm as follows.

\begin{defi}
	Let $(\vphi_x)_{x \in I} \sse \calH$ be a dictionary with synthesis operator $D$. For $v \in \calH$, we define the \emph{atomic norm} $\norm{v}_\calA$ through
	\begin{align}
		\norm{v}_\calA = \left( \inf_{\mu \in \calM(I)} \norm{\mu}_{TV} \st D\mu= v \right), \tag{$\calP_\calA$}
	\end{align}
	with the convention that $\norm{v}_\calA =\infty$ if $v \notin \ran D$.
\end{defi}

By using standard methods, one can prove that $(\calP_\calA)$ has a minimizer. Let us for completeness carry out the argument.

\begin{lem} \label{lem:reprAtomicMeasure}
	For every $v \in \calH$ with $\norm{v}_\calA< \infty$, there exists a $\mu \in \calM(I)$ with $v=D\mu$ and $\norm{v}_\calA=\norm{\mu}_{TV}$. In other words, $(\calP_\calA)$ has a minimizer.
\end{lem}
\begin{proof}
 	Let $R \in \R$ with $R> \norm{v}_A$  be arbitrary and consider the function
	\begin{align*}
		a : \set{ \mu : D\mu = v , \norm{\mu}_{TV} \leq R} \to \R, \mu \to \norm{\mu}_{TV}.
	\end{align*}
	Due to the fact that $D$ is weak-$*$-weak continuous (it is strongly continuous and linear), $a$ is defined on a weak-$*$-compact domain (we invoked the Banach-Alaoglu theorem). $\mu$ is further weak-$*$-lower semi-continuous, meaning that it is continuous if $\R$ is equipped with the topology $\tau_{>} := \set{ (a, \infty), a \in \R \cup \set{-\infty}}$. To see this, we need to prove that for every $a \in \R$,
	\begin{align*}
		\set{\mu \in \calM(I) \, \vert \, \norm{\mu}_{TV}>a} = \set{ \mu \ \vert \ \exists v \in \calC_0(I) : \abs{\sprod{v,\mu}} >a} = \bigcup_{ v \in \calC_0(I)} \set{ \mu \ \vert  \ \abs{\sprod{v,\mu}} >a}
	\end{align*}
	is weak-$*$ open. This is however clear: By definition, $\set{\mu \  \vert \ \abs{\sprod{f,\mu}} >a}$ is  weak-$*$-open, and arbitrary unions of opens sets are open.
	
	We can conclude that the image of $a$ is compact in $\tau_{>}$. The compact sets in $\tau_{>}$ are however exactly the sets $S \sse \R$ with $\inf S \in S$. Hence, $a$ has a minimizer, which is exactly what we need to prove.
\end{proof}

In the following, we will call a measure $\mu_*$ which obeys $D\mu_* = v$ and $\norm{\mu_*}_{TV}= \norm{v}_{\calA}$ an \emph{atomic decomposition} of $v$. Note that we do not assume that this atomic decomposition is unique. Also note that there is no guarantee that $\mu_*$ is the in some sense sparsest measure satisfying $D\mu_* = v$, but rather the one with the smallest $TV$-norm.

\begin{rem} \label{rem:geomAtom}
	In the upcoming paper \cite{infDimCS}, a proof will be presented that one can alternatively define the atomic norm as
	\begin{align}
		\norm{v}_A= \inf \set{t \ \vert \ v \in t \clonv \widetilde{\Phi}}, \label{eq:geomAtom}
	\end{align}
	where $\widetilde{\Phi}$ is the \emph{extended dictionary}  $(\widetilde{\vphi}_{x,\omega})_{x \in I, \omega \in \sph}$, $\widetilde{\vphi}_{x,\omega}=\omega \vphi_x$. This shows that our concept of an atomic norm is very tightly related to the one defined in \cite{chandrasekaran2012convex}. 
	
	In the particular case of $\K=\R$ , $\sph = \set{\pm 1}$, i.e. for symmetric real dictionaries, the definitions  coincide. Also assuming that the dictionary is finite, this fact provides us with an intuitive explanation of the meaning of an atomic decomposition: It reveals on which face of the polytope $\norm{v}_\calA \overline{conv} \Phi$ $v$ is lying. (see Figure \ref{fig:AtomDecomp}). )

\end{rem}

\begin{figure}
	\centering
	\includegraphics[scale=.45]{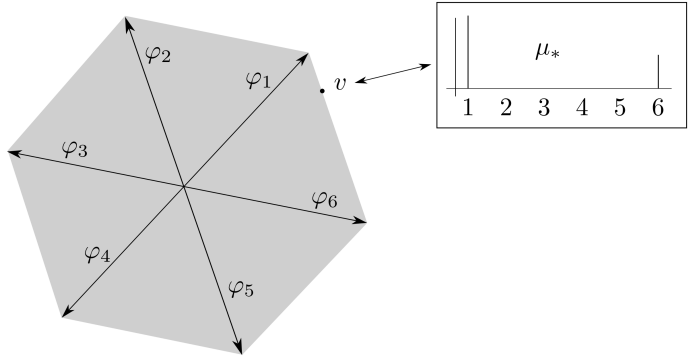}
	\caption{The signal $v$ lies on the face of $\norm{v}_\calA \cdot \clonv \Phi$ which is spanned by $\vphi_1$ and $\vphi_6$. Consequently, it has an atomic decomposition supported on $\set{1,6}$. \label{fig:AtomDecomp}}
\end{figure}

\subsection{Compressed Sensing in Infinite Dimensions}

Now that we have defined atomic norms for (possibly) infinite dimensional signals, let us briefly describe how they could be used to recover structured signals. A \emph{structured signal} is thereby a signal $x_0 \in \calH$ which has a sparse representation in a dictionary $(\vphi_x)_{x\in I}$:
\begin{align*}
	x_0 = \sum_{i=1}^s c_{x_i} \vphi_{x_i} = D\left(\sum_{i=1}^s \delta_{x_i}\right).
\end{align*}
We now take finitely many linear measurements of $x_0$ using the operator $A: \calH \to \K^m$ and want to recover $x_0$ from $b=Ax_0$. The method we would like to consider for this task is
\begin{align}
	\min \norm{x}_\calA \st Ax=b. \tag{$\calP_\calA$}
\end{align}
Due to Lemma \ref{lem:reprAtomicMeasure}, this minimization program is equivalent to 
\begin{align*}
	\min \norm{\mu}_{TV} \st AD\mu=b \tag{$\calP_D$}.
\end{align*}
Assuming continuity of $A$, we can by invoking Lemma \ref{lem:reprAtomicMeasure} (using $(A\vphi_x)_{x\in I}$ as our dictionary) that $(\calP_D)$ and therefore also $(\calP_\calA)$ always has a solution. Using standard methods of convex analysis, one can derive optimality conditions for $(\calP_\calA)$. Since we will carry out similar arguments for slightly more complicated setups in the next section, we choose to not include this. \newline 

The following observation is not hard to make, but will be important:

\begin{lem}
	For a dictionary $(\vphi_x)_{x \in I} \sse \calH$ and measurement operator $A: \calH \to \K^m$, define an \emph{extended dictionary} $(\widetilde{\varphi}_{x, \omega})_{x \in I, \omega \in \sph}$ ($\sph \sse \C$ denotes the unit circle)  through
	\begin{align*}
		\widetilde{\vphi}_{x, \omega} = \omega \vphi_x
	\end{align*} and let $\widetilde{A} : \calH \to \R^{2m}, x \mapsto (\re(Ax), \im(Ax))$. Then the problems $(\calP_D)$ and 
	\begin{align}
		 \min_{\widetilde{\mu}\in \calM(I\cap \sph)} \widetilde{\mu}(I \times \sph)  \st \widetilde{\mu} \geq 0 , \widetilde{A}\widetilde{D}\widetilde{\mu} = (\re(b), \im(b)), \tag{$\calP_D^+$}
	\end{align}
	where $\widetilde{D}$ denotes the synthesis operator of the extended dictionary, have the same optimal values.
\end{lem}
\begin{proof}
	Let us first note that the distinction between $A$ and $\widetilde{A}$ causes no problems - $Av =b$ if and only if $\widetilde{A}v= (\re(b), \im(b))$. Also note that since the measures in $(\calP_D^+)$ are assumed to be positive, they obey $\norm{\widetilde{\mu}}_{TV} = \widetilde{\mu}(I \times \sph)$.

	Now let $\mu$ be a feasible point for $(\calP_D)$. Let $\mu = v \abs{\mu}$ be the \emph{polar decomposition} \cite[p.124, Thm. 6.12]{Rudin} of $\mu$ and define $\widetilde{\mu}$ through
	\begin{align*}
		\widetilde{\mu}(A) = \int_{I \times \sph} \ind_A(x, \omega) d \delta_{ v(x)}(\omega) d \abs{\mu}(x).
	\end{align*}
	Then $\widetilde{\mu} \geq 0$, $\norm{\widetilde{\mu}}_{TV}=\norm{\mu}_{TV}$ and $\widetilde{D}{\widetilde{\mu}}=D\mu$, since
	\begin{align*}
		\sprod{D\mu, v} &= \int_I \sprod{v, \vphi_x} d\mu(x) = \int_I\sprod{v, \vphi_x} v(x) d\abs{\mu}(x) \\
		&= \int_{I} \int_{\sph} \sprod{v, \omega \vphi_x} d \delta_{ v(x)}(\omega) d \abs{\mu}(x) = \int_{I \times \sph} \sprod{v, \widetilde{\vphi}_{x,\omega}} d\widetilde{\mu}(\omega,x).
\end{align*}	 
This proves that the optimal value of $(\calP_D^+)$ is smaller than the one of $(\calP_D)$.
	
	Now let $\widetilde{\mu}$ be a feasible point for $(\calP_D^+)$. Define $\mu$ through
	\begin{align*}
		\mu(A) = \int_I\ind_A(x) \int_{\sph} \omega d\widetilde{\mu}(\omega,x).
	\end{align*}
	Then $\norm{\widetilde{\mu}}_{TV}\geq\norm{\mu}_{TV}$, since
	\begin{align*}
		\norm{\mu}_{TV} = \sup_{\substack{f \in \calC_0(I)\\ \norm{f}_\infty \leq 1}} \abss{\int f(x) d\mu(x)} = \sup_{\substack{f \in \calC_0(I)\\ \norm{f}_\infty \leq 1}} \abss{\int_I \int_{\sph} \omega f(x) d\widetilde{\mu}(\omega,x)} \leq \sup_{\substack{\widetilde{f} \in \calC_0(I\times \sph)\\ \norm{\widetilde{f}}_\infty \leq 1}} \abss{\int_{I \times \sph} \widetilde{f}(\omega,x) d\widetilde{\mu}(\omega,x)} = \norm{\widetilde{\mu}}_{TV}.
	\end{align*}
	Also, $\widetilde{D}\widetilde{\mu}=D\mu$, since
	\begin{align*}
		\sprod{\widetilde{D}\widetilde{\mu},v} = \int_{I \times \sph} \sprod{v, \widetilde{\vphi}_{x,\omega}} d\widetilde{\mu}(\omega, x) = \int_I \sprod{v, \vphi_x} \int_{\sph_1}\omega d\widetilde{\mu}(\omega, x) = \int_I \sprod{v, \vphi_x} d\mu(x) = \sprod{D\mu, v}.
	\end{align*}
 This implies that the optimal value of $(\calP_D)$ is smaller than the one of $(\calP_D^+)$. The proof is finished.
\end{proof}

From the proof of the last lemma, the following result follows. The convenience of it will become clear in the sequel.
\begin{cor} \label{cor:PolarPDplus} If $\mu_*$ is a solution of $(\calP_D)$, then the measure $\widetilde{\mu}_*$ defined through
\begin{align*}
	\widetilde{\mu}_*(A) =\int_{I \times \sph} \ind_A(\omega,x) d \delta_{v(x)}(\omega) d\abs{\mu}_*(x)
\end{align*}
is a solution of $(\calP_D^+)$.
\end{cor}

\section{A Soft Certificate Condition} \label{sec:main}

As was discussed in the introduction, we will now develop a theory for soft recovery using general atomic norms in infinite dimensions. The argument is similar to the one carried out in \cite{Flinth2016Soft}, but there are many subtleties to be taken care of when going over to the infinite dimensional setting.

The heart of the argument will be the fact that if $\mu_*$ is a solution of $(\calP_D)$, the measure $\widetilde{\mu}_*$ defined in Corollary \ref{cor:PolarPDplus} is a solution of  $(\calP_D^+)$ restricted to measures supported on $\supp \widetilde{\mu}_*$. The optimality conditions of the latter problem are particularly nice, and they will allow us to prove properties of the minimizer $\mu_*$. In order to derive said conditions, we will use the duality theory of convex programs in infinite dimensions developed in \cite{BorweinLewis1992}.

 The theorem we will apply uses the notion the \emph{quasi relative interior} of $\qri(C)$ a convex subset $C$ of a locally convex vector space $X$. The authors of \cite{BorweinLewis1992} define it as the set of vectors $x$ in $C$ so that the closure of the cone generated by $x-C$, $\cone(x-C) = \set{\lambda(x-c) \ \lambda >0, c \in C}$  is a subspace. The quasi relative interior is not necessarily equal the \emph{relative interior} $\text{ri}(C)$, which is the interior of $C$ when regarded as a subset of the affine hull of $C$. Another notion the theorem uses is the one of \emph{proper} convex functions $f: X \to \R \cup {\infty}$. A convex function is proper if it has a non-empty domain (i.e. pre-image of $\R$). 
 
\begin{theo} (Simplified version of \cite[p.34, Cor. 4.6]{BorweinLewis1992}) \label{th:duality}
	Let $X$ be a locally convex real vector space, $g: X \to \R \cup \set{\infty}$, $h: \R^m \to \R \cup \set{\infty}$ be proper convex functions  and $M: X \to \R^m$ be linear. Suppose that
	\begin{align*}
		M\qri(\dom g) \cap \ri(\dom h) \neq \emptyset.
	\end{align*}
	Then the optimal values of the following two programs are equal:
	\begin{align}
		\inf_{x\in X}  \, &g(x) + h(Mx), \tag{$\calP$} \\
		 \sup_{\lambda \in \R^m}  \, &g^*(M^*\lambda) + h^*(\lambda). \tag{$\calD$}
	\end{align}
	Thereby $f^*$ denotes the convex conjugate of $f$ \cite[p.102f]{Rockafellar1970}.
\end{theo}

We would now like to apply this theorem to $(\calP_D^+)$  restricted to measures supported on some set $I_*$. Here, $X$  is the space of real Radon measures on $I_*$, and $g$ and $h$ are defined as follows:
\begin{align*}
	g(\mu) = \begin{cases} \int_{I_*} d\mu  &\text{ if } \mu \geq 0 \\
							\infty &\text{ else} \end{cases}, \quad 
	h(p) = \begin{cases} 0 &\text{ if } p=\widetilde{b} \\
							\infty &\text{ else,} \end{cases}
\end{align*}
where we defined $\widetilde{b}=(\re(b), \im(b))$.
These are both proper. Further, $\ri(\dom h) = \{\widetilde{b}\}$ and $\qri(\dom g) = \set{ \mu \ \vert \ \mu \geq 0, \supp \mu = I_*}$ -- see \cite[p. 28, Ex. 3.11]{BorweinLewis1992}\footnote{Note that the therem in the reference given technically only applies to the case that $I_*$ is compact Hausdorff space. The compactness is however only used to be able to apply Urysohn's Lemma, which holds on locally compact metric spaces as well.}. Their convex conjugates are given by
\begin{align*}
	g^*(\psi) = \begin{cases} 0  &\text{ if } \sup \psi \leq 1 \\
							-\infty &\text{ else} \end{cases}, \quad 
	h(p) =\langle \widetilde{b},p \rangle.
\end{align*}
We conclude that the dual program of $(\calP_D^+)$ restricted to $I_*$ is 
\begin{align}
	\sup_{p \in \R^m} \langle b, p\rangle  \st \sup_{z \in I_*} ((\widetilde{A}\widetilde{D})^*p)(z) \leq 1. \tag{$\calD_D^+$}
\end{align}

We have now presented all the theory we need in order to prove the soft recovery condition. We will make a structural assumption on the dictionary which will come in handy, namely that it consists of $r$ subdictionaries $\Phi_j$ (note that $r=1$ is possible, so this is not per se restricting the dictionaries we can consider). Let us further introduce the notation:
\begin{align*}
	\norm{v}_{\calA_j}^* = \sup_{\vphi \in \Phi_j} \abs{\sprod{v, \vphi}}.
\end{align*}
Now we are ready to formulate and prove our main result. 

\begin{theo} \label{th:soft} Let $(\vphi_x)_{x \in I}$ be a normalized dictionary for $\calH$ and $A: \calH \to \K^m$ be continuous. Let further $v_0 \in \calH$ be given through
\begin{align*}
	v_0 =c_{x_0}^0\delta_{x_0} + D(\mu^c)
\end{align*}
for a complex scalar $c_{x_0}^0$ and a measure $\mu^c$ such that  $\norm{v_0}_\calA= \norm{ c^0_{x_0}\delta_{x_0}+\mu^c}_{TV} = 1$.

Now let $j_0 \in \set{1, \dots, r}$, with $x_0 \in \calA_{j_0}$, and $\sigma\geq 0, t \in (0,1]$. Suppose that there exists a $\nu \in \ran A^*$ with
\begin{align}
	\int_I \sprod{\vphi_x, \nu} d(c_{x_0} \delta_{x_0} + \mu^c)  &\geq 1 \label{eq:Ankare} \\
	\norm{\nu}_{\calA_{j}}^* &<1, \quad j \neq j_0 \label{eq:otherSubs} \\
	\abs{\sprod{\nu, \vphi_{x_0}}} &\leq \sigma \label{eq:atPoint} \\
	\norm{\Pi_{\sprod{\vphi_{x_0}}^\perp}\nu}_{\calA_{j_0}} &\leq 1-t. \label{eq:orthCompSameSub}
\end{align}

Then the following is true: If $\mu_*$ is a solution of $(\calP_D)$, there exists an $x \in \supp \mu_*$ with $\vphi_x \in \Phi_{j_0}$ such that
\begin{align} \label{eq:scalProd}
	\abs{\sprod{\vphi_x, \vphi_{x_0}}} \geq \frac{t}{\sigma} .
\end{align}

\end{theo}

\begin{proof}
	Let $\widetilde{\mu}_*$ be the measure on $I \times \sph$ defined in Corollary \ref{cor:PolarPDplus}, which by the same corollary is a solution of $(\calP_D^+)$, in particular when it is restricted to measures supported on $$I_*=\supp \widetilde{\mu}_* = \overline{\set{(v(x),x) \ \vert \ x \in \supp \mu_*}},$$
	where $v \cdot \abs{\widetilde{\mu}_*}$ is the polar decomposition of $\widetilde{\mu_*}$. Theorem \ref{th:duality} applies to this program: $\widetilde{A}\widetilde{D} \qri( \dom g) = \widetilde{A} \set{\mu \ \vert \supp \mu = I_*}$ has a non-empty intersection with $\ri(\dom(h))= b$, since $\widetilde{A} \widetilde{D}\widetilde{\mu}=b$. Hence, the optimal value of $\calD_D^+$ is equal to $\norm{\mu_*}_{TV}$. 
	
	Due to the fact that $\mu_0= \delta_{x_0} + \mu_c$ is a feasible point for the original problem $\calP_\calA$, $\norm{\mu_0}_{TV} \leq \norm{ c_{x_0}^0 \delta_{x_0} + \mu_c}_{TV} =1$. This implies that if $\langle \widetilde{b},p \rangle \geq 1$ for some $p$, then necessarily $\sup_{(x, \omega) \in I_*} ((\widetilde{A}\widetilde{D})^* p )(\omega,x)\geq 1$: If $\sup_{(x, \omega) \in I_*} ((\widetilde{A}\widetilde{D})^* p )(\omega,x)< 1$, we could by multiplying $p$ by a positive scalar slightly larger than one construct a vector $p'$ which is a feasible point for $(\calD_{D}^+)$ but have a value of $\langle \widetilde{b}, p' \rangle >1$, which would contradict the fact that the optimal value of $(\calD_D^+)$ is at most $1$.
	
	Due to the fact that $\nu \in \ran A^*$, there exists a $\lambda \in \K^m$ with $\nu = A^*\lambda$. Define $p= (\re(\lambda), \im(\lambda)) \in \R^{2m}$. Then for $v \in \calH$ arbitrary
	\begin{align}
		\langle v, \widetilde{A}^*p \rangle = \langle{\widetilde{A}v,p}\rangle = \sprod{\re(Av),\re{\lambda}} + \sprod{\im(Av),\im(\lambda)} &= \re( \sprod{v, \re(\lambda)+ i \im(\lambda)} ) \nonumber \\
		&= \re(\sprod{v, A^*\lambda}) = \re(\sprod{v, \nu}). \label{eq:nuVsP}
	\end{align}
	This implies that, since $\widetilde{b}= \widetilde{A}v_0$
	\begin{align*}
		\sprod{b,p} = \langle v_0, \widetilde{A}^*p \rangle =  \re(\sprod{v_0, \nu} ) =  \int_I \re(\sprod{\vphi_{x},\nu}) d(c_{x_0}^0\delta_{x_0} + \mu_c)\geq 1,
	\end{align*}
	where we applied \eqref{eq:Ankare}. We conclude that there exists an $(\omega, x) \in I_*$ with $((\widetilde{A} \widetilde{D})^*p)(\omega,x)\geq 1$ (note that the $\sup$ is really a $\max$ due to $(\widetilde{A} \widetilde{D})^*p \in \calC_0(I \times \sph)$). Note that $((\widetilde{A} \widetilde{D})^*p)(\omega,x)\geq 1$ can be equivalently restated as
	\begin{align}
		1 \leq \langle \delta_{\omega} \otimes \delta_x,(\widetilde{A} \widetilde{D})^*p \rangle  = \langle  \widetilde{D}\delta_{\omega} \otimes \delta_x,\widetilde{A}^*p \rangle = \re( \sprod{\omega \vphi_x, \nu}). \label{eq:prelargeScalProd}
	\end{align}
	We used \eqref{eq:nuVsP}. The form of $I_*$ implies that $x \in \supp \mu_*$. Due to \eqref{eq:otherSubs}, this $x$ must have the property $\vphi_x \in \calA_{j_0}$. Now we apply \eqref{eq:atPoint} and \eqref{eq:orthCompSameSub} to derive
	\begin{align*}
		1 \leq \re( \omega \sprod{\vphi_x, \vphi_{x_0}}\sprod{\vphi_{x_0}, \nu} + \omega \langle{\Pi_{\sprod{x_0}^\perp} \vphi_{x_0}, \nu}\rangle) \leq \abs{\sprod{\vphi_{x_0},\nu}} \abs{\sprod{\vphi_x, \vphi_{x_0}}} + \norm{\Pi_{\sprod{x_0}^\perp}\nu}_{\calA_{j_0}}^* \leq \sigma \abs{\sprod{\vphi_x, \vphi_{x_0}}} + (1-t).
	\end{align*}
	Rearranging brings the result.
\end{proof}   

Let us make a few comments yields moving on.

\begin{itemize}

\item In the following, we will sometimes call a vector $\nu$ fulfilling \ref{eq:Ankare}-- \ref{eq:orthCompSameSub} a \emph{soft certificate}.

\item 
We assumed that  the ground truth $v_0$ has an atomic decomposition of the form $c_{x_0}^0 \delta_{x_0} + \mu^c$. Hence, the atomic decomposition has a peak at $\delta_{x_0}$ of relative power $c_{x_0}^0$. $\mu^c$ was not assumed need to have a certain structure. In particular can be a train of $\delta$-peaks supported on other indices $(x_i)_{i=1}^s$, but it can also be more irregular. 

Also note that assuming that $\norm{c_{x_0}^0 \delta_{x_0} + \mu^c}_{TV}=1$ merely is a matter of normalization.

\item The condition of the theorem depends on $x_0$. It hence has to be applied for each index separately. The previous point however makes it clear that $v_0$ per se does not have to have a $1$-sparse atomic decomposition. Also, the fact that the theorem is to be applied to each index separately is not per se a weakness -- this we will see in Section  \ref{sec:sep}.

\item The theorem does not provide any statements on how the measure behaves more than that it will be supported in a point $x$, which in the sense of \eqref{eq:scalProd} is close to $x_0$. In particular, it does not only need to contain peaks, and can have many other 'spurious' peaks.
	
	For the sake of signal recovery, the information that the theorem gives us can still be put to good use: After having solved $(\calP_\calA)$ and checking where the support of the solution measure $\mu_*$ are located in the solution, we can choose points $\set{x_i}$ and restrict our search to signals which are linear combinations of the corresponding dictionary elements $\set{\varphi_{x_i}}$. Since these dictionary elements necessarily need to be close, solving the restricted (finite-dimensional) have better chances at succeeding to recover $v_0$ than the original problem.
\end{itemize}

In the remainder of the paper, we will see that the relatively abstract main result can be successfully applied in quite a few different concrete settings.

\section{Applications} \label{sec:Appl}

In this section, we discuss four different cases of atomic norm minimization, and which impact the main result has in each of those. In most parts, we will stay in the setting that $\calH$ is finite-dimensional, the reason being that we want to derive concrete recovery results using randomly generated measurement matrices. In the final part however, we will also consider an infinite dimensional example.

\subsection{$\ell_{12}$-minimization.} \label{sec:L12}
 The concept of soft recovery was first introduced by the author in \cite{Flinth2016Soft}. There, it was only treated in the context of $\ell_{12}$-minimization. The argument in that paper had the same flavour, but did not utilize the concept of dictionaries and atomic norms to the same extent. Let us briefly check back-combatibility in the meaning that the old result qualitatively is a special case of the result in this papear
 
 For a matrix $X \in \R^{k,d}$ with columns $X_i$, $i \in 1, \dots, d$, the $\ell_{12}$-norm is defined as the sum of the $\ell_2$-norms of its columns, i.e.
 \begin{align*}
 	\norm{X}_{12} = \sum_{i=1}^d \norm{X_i}_2.
 \end{align*}
Consider the dictionary
 \begin{align*}
 	\Phi^{12}:= \set{ \vphi_{i,\eta} = \eta e_i^* \ \vert \ \eta \in \sph^{k-1}, i \in [d]}. 
 \end{align*}
 It is not hard to convince oneself that this is a dictionary for $\R^{k,d}$, when $I = [d] \times \sph^{k-1}$ is equipped with the obvious metric. Furthermore, the $\ell_{12}$-norm is the atomic norm with respect to this dictionary, and the atomic decomposition is given through
 \begin{align*}
 	X = \sum_{i=1}^d \norm{X_i}_2 \cdot \eta_i e_i^*,
\end{align*}  
where $\eta_i$ are the normalized columns of $X$ (for a rigorous proof of this, see Section \ref{sec:L12Proof}.) Note that $\Phi^{12}$ naturally can be decomposed into the $d$ subdictionaries $\Phi^{12}_j = \set{\eta e_j^* \ \vert \ \eta \in \sph^{k-1}}$, and that $\norm{X}_{\calA_j} = \norm{X_j}_2$. Therefore, for an index $i_0$ and corresponding normalized column $\eta_{i_0}$, we have 
\begin{align*}
	X^0= \norm{X^0_{i_0}}_2 \eta_{i_0} e_{i_0}^* + D\left( \sum_{i \neq i_0} \norm{X_i}_2\delta_{\eta_i e_i^*}\right)
\end{align*} and \eqref{eq:Ankare}-\eqref{eq:orthCompSameSub} reads
\begin{align}
	\sum_{i \in [d]} \norm{X}_i \sprod{\eta_i, V_i} &\geq 1, \qquad  \norm{V_j}  <1 , \quad j \neq i_0   \label{eq:newfirst} \\
	\abs{\sprod{V_{i_0},\eta_{i_0}}} &\leq \sigma , \qquad \norm{ \Pi_{\sprod{\eta_{i_0}}^\perp }V_{i_0}}_2 \leq t \label{eq:newsecond}
	\end{align}
where $V$ is the $\nu$-vector from Theorem \ref{th:soft}. The first row of these conditions are exactly as in \cite{Flinth2016Soft}, whereas the second row is typologically different: In the previous paper, the corresponding conditions are \begin{align}
\norm{V_{i_0}}_2\leq \cos(\alpha)^{-1},  \qquad \angle\left(\eta_{i_0}, V_{i_0}\right) \leq \alpha, \label{eq:oldsecond}
\end{align}
where $\angle(x,y)$ denotes the angle between the vectors $x$ and $y$. 

Let us argue that the old condition qualitatively is the same as the new one. The key is to observe that \eqref{eq:newfirst} implies that $\sprod{\eta_{i_0}, V_{i_0}} > 1$. On this set, \eqref{eq:newsecond} and \eqref{eq:oldsecond} describes similar sets for $\sigma= \cos(\alpha)^{-1}$ and $t=\tan\alpha$, as can be seen in Figure \eqref{fig:NewVsOld}.

\begin{figure}
\centering
\includegraphics[scale=.8]{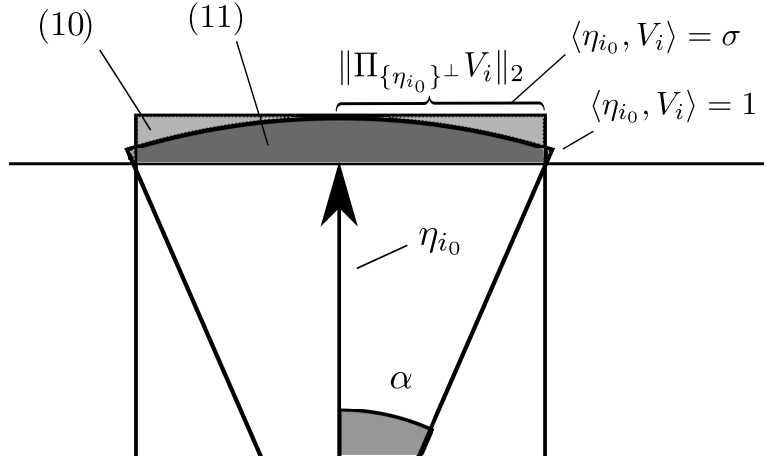}
\caption{A graphical comparison between the sets generated by \eqref{eq:newsecond} and \eqref{eq:oldsecond}, respectively, on the set $\sprod{\eta_{i_0},V_{i_0}}>1$. \label{fig:NewVsOld}}
\end{figure}

The theory in the article \cite{Flinth2016Soft} guarantees that the angle between the $i_0$:th row of a solution to the $\ell_{12}$-minimization is smaller than $2\alpha$. Theorem $\ref{th:soft}$ instead ensures the existence of a dictionary element of the form $\eta e_j^*$ with $\abs{\sprod{\eta_{i_0}e_{i_0}^*, \eta e_{i_0}^*}}  = \abs{\sprod{\eta_{i_0}, \eta}} \geq \tfrac{\sigma}{t}$. This is, \emph{after identification of antipolar vectors}, more or less an equivalent condition. 

Note that the difference between appearance of $\eta e_{i_0}^*$ and $- \eta e_{i_0}^*$ in the atomic decomposition of $X^*$ is subtle: By adjusting the sign of the corresponding coefficient in the atomic decomposition, we see that there will \emph{exist} a solution of $(\calP_D)$ with $(\eta, i_0) \in \supp \mu$ so that $\sprod{\eta, \eta_{i_0}} \geq 0$.

\subsection{Component Separation from Low-Dimensional Sketches} \label{sec:sep}
 An intriguing application of classical compressed sensing is that of \emph{geometric separation}: Using $\ell_1$-minimization, one can, for instance, separate dendrites from cell nuclei in images of neurons, or extract bright stars from an image of a galaxy.
The idea behind this technique can be described as follows. The ground truth signal $x_0$ is assumed to consist of two components $x_0^1$ and $x_0^2$.  Each component is assumed to possess a sparse decomposition in a distinct system -- say $x_0^1 = \Psi c^1$ and $x_0^2=\Theta c^2$, with $\Psi$ and $\Theta$ matrices and $c^i$ sparse. Although it is possible to consider more general settings \cite{KutyniokLim2010} and the problem can be modelled most faithfully in an infinite dimensions \cite{donoho2013microlocal}, we will here for simplicity assume that $\Psi$ and $\Theta$ are orthonormal systems in a finite dimensional Hilbert space. In order for the separation to work, $\Psi$ and $\Theta$ need to be morphologically different -- one could for instance consider $\Psi$ being an orthonormal system of wavelets, or spikes, and $\Theta$ to be the Fourier system. A mathematical way to measure the dissimilarity between the system is the \emph{mutual coherence}
\begin{align*}
	\kappa(\Psi, \Theta) =\max_{\psi \in \Psi, \theta \in \Theta} \abs{\sprod{\psi,\theta}}.
\end{align*}
A low coherence is interpreted as a high dissimilarity between the systems. \newline

Now let us assume that we only have access to a low-dimensional sketch $b=Ax_0 \in \C^m$ of the signal $x_0 \in \C^n$. Concerning the fact that the vector $[c^1, c^2]$ is sparse, it should be possible to recover the components $c^i$ using the following program:
\begin{align}
	\min \norm{c^1}_1 + \norm{c^2}_1 \st A(\Psi c^1 + \Theta c^2) = b, \tag{$\calP_{\text{sep}}$}
\end{align}
which of course is equivalent to minimization of the atomic norm with respect to the dictionary $\Phi= \Psi \cup \Theta$. Using standard techniques of compressed sensing, it is not hard to derive that, assuming $A$ is drawn from a reasonable probability distribution and $m$ is larger than some threshold which is essentially proportional to the sparsities of the $c^i$, the program $(\calP_{\text{sep}})$ will recover the components $c^i$ with high probability. In the case that $A=\id$, the problem was solved already in one of the pioneering mathematical papers on sparse recovery \cite{donoho2001uncertainty}. Theorem~7.1 of said paper states that if $
	\frac{1}{2} \left(\kappa(\Psi, \theta)^{-1} + 1\right) \geq s$, then any $s$-sparse $[c^1,c^2]$ will be recovered by solving $(\calP_{\text{sep}})$. This bound was later improved in \cite{elad2002generalized}, where it is shown that actually
\begin{align}
	\kappa(\Psi, \theta)^{-1}\left( \sqrt{2}-1 + \frac{1}{2}\right) \geq s, \label{eq:classCond}
\end{align}
suffices.

The case that $A \neq \id$ was treated in \cite{rauhut2008compressed}. The authors of that paper assume that $A \in \R^{m,n}$ is random and obeys a certain concentration of measure property. They prove (Corollary 4.1 of that paper)  that provided $\kappa(\Psi, \theta)^{-1}(s - 1) \geq \tfrac{1}{16}$ and $m \geqsim s \log(n)$, $(\calP_{\text{sep}})$ will recover the correct $c$. NHote that the result also applies to more general dictionaries, not only concatenation of two orthonormal bases. \newline

Suppose now that one of the coefficients, say $c^1_{i_0}$, is considerably larger the others. Can we maybe derive a guarantee for the program $(\calP_{\text{sep}})$ detecting that this coefficient is non zero using less measurements than for recovering the components exactly? This would be relevant for applications: In the galaxy imaging application, it maybe is only of interest to detect the position of a very bright star, instead of also resolving the underlying galaxy. Using Theorem \ref{th:soft}, this is in fact possible.

\begin{prop} \label{prop:sep} Let $A$ be a measurement matrix whose rows are i.i.d. copies of a random vector $X$  with the following properties:
\begin{enumerate}[(i)]
	\item {\bf Isotropy:} $\erw{X X^*} = \id$.
	\item {\bf Incoherence to $\Psi$:} $\sup_{\psi \in \Psi} \abs{\sprod{X, \psi}} \leq M$ almost surely.
	\item {\bf Incoherence to $\Theta$:} $\sup_{\theta \in \Theta} \abs{\sprod{X, \theta}} \leq M$ almost surely.
\end{enumerate} 
 Suppose that $x_0 = \Psi c^1 + \Theta c^2$ with $\norm{c^1}_1 + \norm{c^2}_1 =1$, and let $i_0 \in \supp c^1$. If there exists a $\gamma>0$ with
 \begin{align}
 	\kappa \cdot \frac{4}{\abs{c_{i_0}^1}} \leq 1 -\gamma, \label{eq:CoherenceCond}
 \end{align}
 where $\kappa= \kappa(\Psi, \Theta)$, the following holds: If $m=p\cdot r$ with
\begin{align*}
	r \geq \log\left( \frac{8}{\gamma}\right) \text{ and }  p \geqsim \left(C_\gamma^2\frac{M^2}{\kappa^2} + \frac{C_\gamma (M^2+\kappa)}{\kappa} \right) \log\left(\frac{nr}{\delta} \right),
\end{align*}
where $C_\gamma= \tfrac{3}{\gamma}\left(1-\tfrac{2\gamma}{3}\right)$, then with probability larger than $1-\delta$, any solution $(c_*^1, c_*^2)$ of $(\calP_{\text{sep}})$ will have the property $i_0 \in \supp c_*^1$.
\end{prop}

The proof, which uses the golfing scheme \cite{gross2011golfing} will inevitably be quite technical and is therefore postponed to  Section \ref{sec:ProofSep}. The main trick is to utilize that if we know that a member $\vphi_k$ of the dictionary $\Phi= \Psi \cup \Theta$ satisfies $\abs{\sprod{\vphi_k, \psi_{i_0}}} > \kappa(\Psi, \Theta)$, then necessarily $\vphi_k = \psi_{i_0}$. \newline

Before ending this subsection, let us make a few remarks concerning Proposition \ref{prop:sep}.
\begin{rem}
\begin{itemize}
	\item The parameter $M$ is larger than $1$. This follows from
	\begin{align*}
		n &= \tr(\id) = \erw{\tr(XX^*)} = \erw{\norm{X}_2^2} \leq n M^2,
	\end{align*}
where we used $\norm{X}_2^2 = \sum_{\psi \in \Psi} \abs{\sprod{X, \psi}}^2 = \sum_{\theta \in \Theta} \abs{\sprod{X, \theta}}^2$, due to the $ONB$ property. 

\item A distribution for which $M$ almost achieves its optimal value is uniform random sampling of a system $(\sqrt{n} \rho_k)_{k=1}^n$, where $R=(\rho_k)_{k=1}^n$ is an $ONB$ with $\kappa(R, \Psi), \kappa(R, \theta) \sim \sqrt{\frac{\log(n)}{n}}$. Such a system can be constructed by drawing an orthogonal matrix at random. In this case, $M$ can be bounded by $\log(n)$ with high probability.

	\item Note that \eqref{eq:CoherenceCond} practically has the same flavour as \eqref{eq:classCond} if we identify $s$ with $\abs{c_{i_0}^1}^{-1}$. If $c$ has $s$ equally large components, we obtain $s=\abs{c_{i_0}^1}^{-1}$. If however one component is unproportionally large, we will have $\abs{c_{i_0}^1}^{-1} \ll s$. Hence, for such components, the statement of Proposition \ref{prop:sep} is stronger than what would be implied by only considering the sparsity of $c$ and \eqref{eq:classCond}. 

	To be concrete, let us consider the case that $n=1000$ and that $c^1, c^2$ have the following form:
	\begin{align*}
		c^1_0 =\frac{1}{2}, \quad c^1_i = \frac{1}{4}\cdot\frac{1}{999}, \quad i \geq 2 \quad, c^2_i = \frac{1}{4}\cdot\frac{1}{1000}, \quad i =1, \dots 100.
\end{align*}	
	$[c^1,c^2]$ is in this case $2000$-sparse, and the above mentioned results from the literature will only provide vacuous statements. In fact, the signal is not even close to being approximately sparse. We have $$\inf_{w \text{ $s$-sparse}} \norm{[c^1,c^2]-w}_1 = \frac{1}{2}-\frac{\min(s,1000)-1}{999}\cdot\frac{1}{4} - \frac{\max(s-1000,0)}{1000}\cdot\frac{1}{4}, $$ hence, it only drops very slowly with $s$. Hence, also theorems applying to such vectors will only yield weak statements.
	For our results however, we only need to note that $(c_{i_0}^1)^{-1} =2$. Hence, the soft recovery of the index $i_0$ will succeed even for dictionaries with coherence bounded by a constant, and $p \geqsim \kappa^{-2}\log(n)$!

 \item As for the number of measurements, the result is presumably not optimal. Considering the previous remark, we can interpret $\kappa^{-1}$ as the \emph{maximal sparsity for which $(\calP_{\text{sep}})$ will work}. Considering that  $M$ is larger than $1$, we see that $\kappa^{-2}$ is the limiting factor , i.e. the 'sparsity squared'. A linear dependence would be more desirable. 
\end{itemize}
\end{rem}

\subsection{Nuclear Norm Minimization} \label{sec:Nuc}
One prominent example of an atomic norm minimization problem is \emph{nuclear norm minimization} for recovery of low-rank matrices:
\begin{align*}
	\min_{X \in \R^{n,k}} \norm{X}_* \st A(X) =A(X_0). \tag{$\cal{P}_*$}
\end{align*} The nuclear norm $\norm{X}_*$  of a matrix $X \in \R^{n,k}$ is thereby defined as the sum of its singular values. This problem has been thoroughly discussed in the literature (see for instance \cite{gross2011golfing, recht2010guaranteed}). 

The nuclear norm is the atomic norm with respect to the dictionary introduced in Example \ref{ex:E}.2 for the case $U= \R^n$ and $V= \R^k$ (a rigorous proof of this can again be found in Section \ref{sec:NucProof}). The atomic decomposition is furthermore given by the $SVD$:
\begin{align*}
	X = \sum_{i=1}^r \sigma_i v_i \otimes u_i,
\end{align*}
where $r = \rank X$. Since the dual norm of the nuclear norm is the spectral norm \cite[Prop 2.1]{recht2010guaranteed}, \eqref{eq:Ankare} - \eqref{eq:orthCompSameSub} take the form (\eqref{eq:otherSubs} is vacuous):
\begin{align}
	\sum_{i=1}^r \sigma_i \sprod{v_i, \nu u_i} \geq 1, \quad \sprod{v_{i_0}, \nu u_{i_0}} \leq \sigma \label{eq:SoftNuc1} \\ 
	 \norm{\nu - \sprod{v_{i_0}, \nu u_{i_0}} v_{i_0} \otimes u_{i_0}}_{2 \to 2} \leq 1-t. \label{eq:SoftNuc2}
\end{align}
We have of course assumed that $\norm{X_0}_*=1$. If there existes a $\nu$  with \eqref{eq:SoftNuc1}--\eqref{eq:SoftNuc2}, the main result implies that one of the singular pairs $(v^*, u^*)$ of the solution $X_*$ of $(\calP_*)$ obeys $$\sprod{v^* \otimes u^*,v_{i_0} \otimes u_{i_0}} = \left(1 - \tfrac{1}{2}\norm{u^* - u_{i_0}}_2^2\right) \left(1 - \tfrac{1}{2} \norm{v^*-v_{i_0}}_2^2 \right) \geq \frac{t}{\sigma}.$$ If $\tfrac{t}{\sigma}$ is close to $1$, this implies that  $\norm{v^*-v_{i_0}}_2$ and $\norm{u^* - v_{i-i_0}}$ have to be small (see Figure \ref{fig:DistVsRatio}). \newline

\begin{figure}
\centering
\includegraphics[scale=.20]{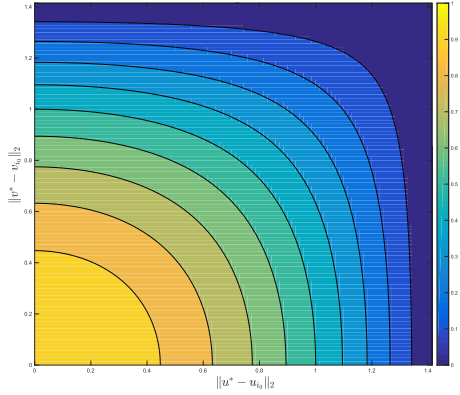}
\caption{Values of $\norm{u^*-u_{i_0}}_2$, $\norm{v^*-v_{i_0}}_2$ vs. the value of $t/\sigma$. \label{fig:DistVsRatio}}
\end{figure}

In order to conceptually prove that it can be easier to find a soft certificate than to find one satisfying the conditions for an exact dual certificate \cite[(2.9) and (2.11)]{recht2010guaranteed}:
\begin{align}
	\nu = \sum_{i=1}^r v_i \otimes u_i + W ,\quad W \vert_{\spn (u_i)_{i=1}^r}=0, \quad \ran W \perp \spn (v_i)_{i=1}^r, \label{eq:ExactNuc}
\end{align}
let us assume that $A: \R^{k,n} \to \R^m$ is a Gaussian map. Then $\ran A^*$ is a uniformally distributed $m$-dimensional subspace of $\R^{k,n}$, and we can apply the powerful machinery of \emph{statistical dimensions} \cite{AmelunxLotzMcCoyTropp2014}. The statistical dimension of a convex cone $C$ is defined as
\begin{align*}
	\delta(C) = \erw{ \norm{\Pi_C g}_2^2},
\end{align*}
where $g$ is a Gaussian and $\Pi_C$ denotes the metric projection onto $C$, i.e.,  $\argmin_{c \in C} \norm{g-c}_2$. Now, a bit streamlinedly put (for details, see Theorem $I$ in the mentioned paper), the main result of the mentioned paper says that if $C$ is a convex cone in a $d$-dimensional real vector space and $m > d-\delta(C)$, then $\ran A^*$ intersects $C$ non-trivially with high probability. 

Due to the linear structure of $\ran A^*$, there exists a $\nu$ satisfying \eqref{eq:SoftNuc1}--\eqref{eq:SoftNuc2}, or \eqref{eq:ExactNuc}, respectively, if and only if there exists a $\nu$ lying in the cones $C_{\text{soft}}$  and $C_{\text{exact}}$ generated by vectors satisfying those constraints. In Section \ref{sec:numest} we describe a very rudamentary procedure to estimate the statistical dimensions of those cones numerically. Figure \ref{fig:dims}  depict the results of applying this method to the nuclear norm setting. We chose to set $n=k=30$, and tested ranks of $X_0$ from $r=1$ to $4$. In the case that $r>1$, we chose $\sigma_1 = \frac{7}{10 }$ and the rest of the nonzero singular values equal to $\tfrac{3}{10(r-1)}$.

\begin{figure}
\centering
\includegraphics[scale=.125]{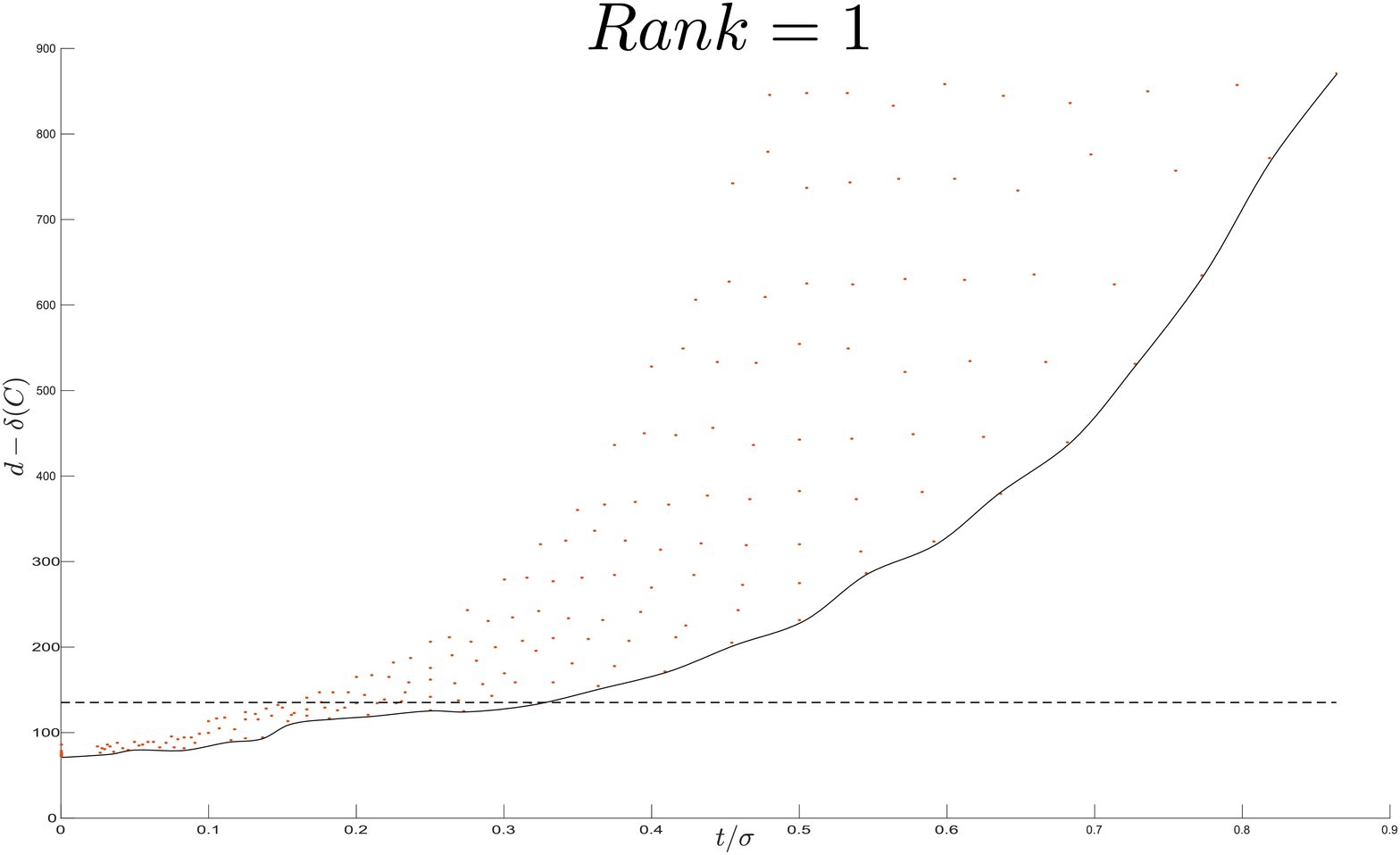} \includegraphics[scale=.125]{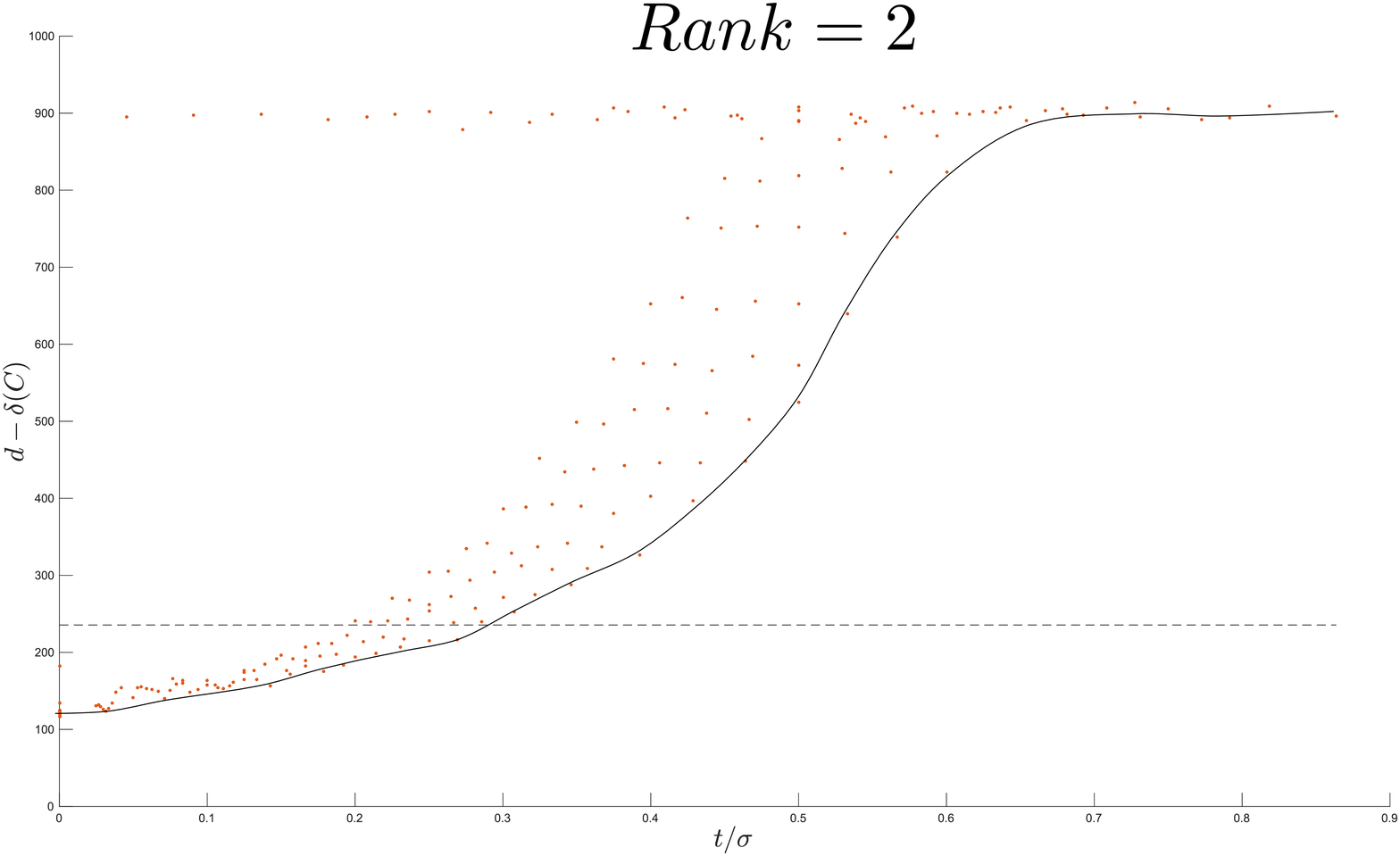} \\
\includegraphics[scale=.125]{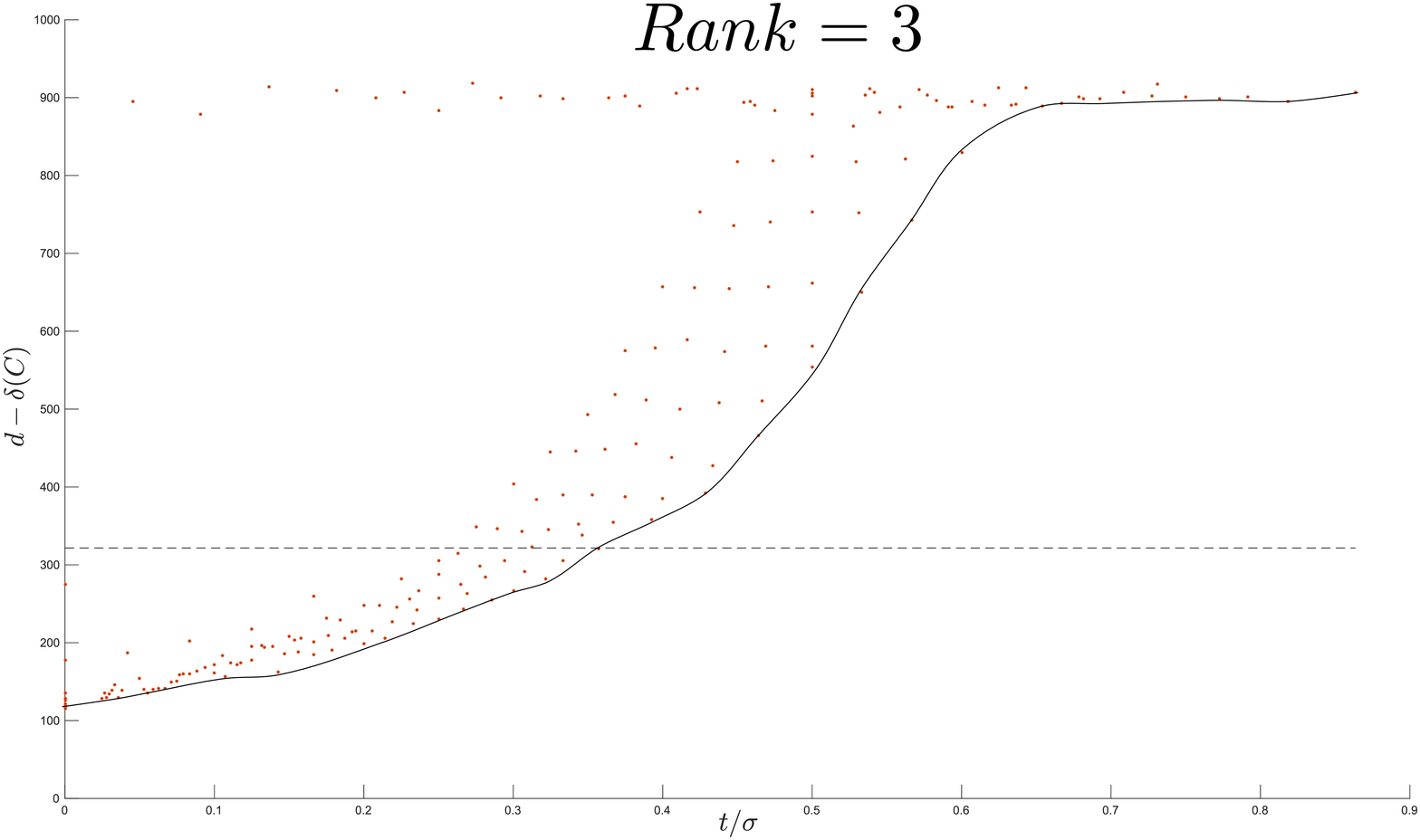} \includegraphics[scale=.125]{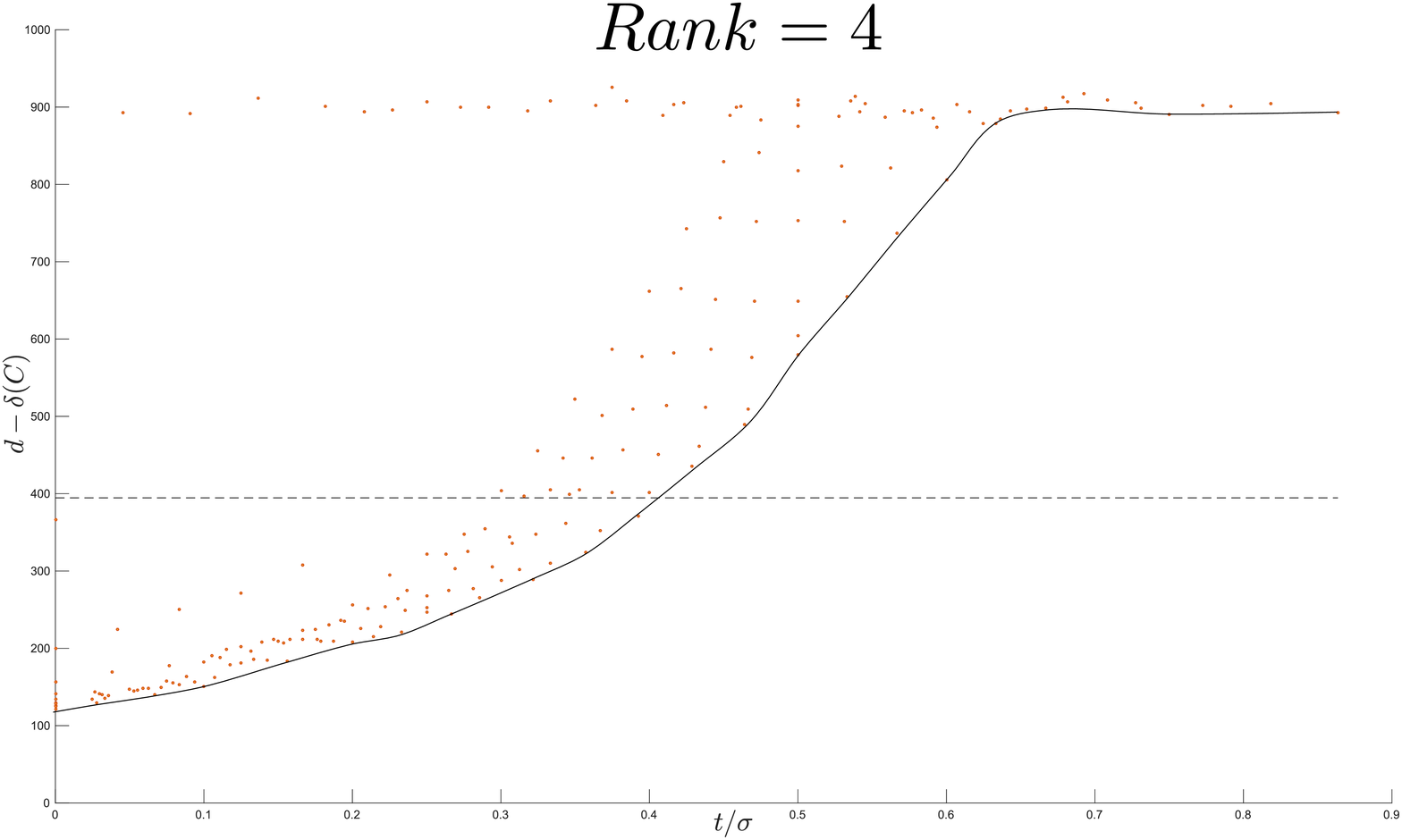}

\caption{Numerically calculated values for $d-\delta(C)$ for different ranks. The dashed line marks the value of $d-\delta(C)$ needed for exact recovery. The dots are the actual datapoints $(t/\sigma, d-\delta(C))$ that were considered, the solid lines are interpolations of those values approximating the optimal value of $d-\delta(C)$ for each value of $t/\sigma$. \label{fig:dims} }
\end{figure}

\subsection{Support Tracking in Superresolution} \label{sec:SuperResolution}
 
 A compressed sensing problem which is best formulated in infinite dimensions is the one of \emph{super-resolution}. Suppose that $\mu_0$ is a finite train of $\delta$-peaks: \begin{align} \label{eq:deltatrain}
 	\mu_0 = \sum_{x \in I_0} c_x^0 \delta_x,
 \end{align}
 where $I_0$ is some finite subset of $\R$. Note that the points $x \in I_0$ are not assumed to lay on some fixed, discretized grid. Now assume that we can only measure $\mu_0$ indirectly by some measurement process $A\mu_0$ and are asked to recover $\mu_0$. Can we do this by solving the following $TV$-minimization program
 \begin{align}
 	\min \norm{\mu}_{TV} \text{ subject to } A\mu =A \mu_0 \tag{$\calP_{TV}$} ?
 \end{align}

This program has already been considered in the literature several times.  The authors of \cite{candes2014towards} prove that if one assumes a separation condition $\inf_{x\neq x' \in I_0} \abs{x-x'} \geq 2/f_c$, then one can choose $A$ as measuring the $2 f_c +1$ first  of the Fourier coefficients. The technique of the paper is to construct an exact dual certificate guaranteeing that $\mu_0$ is the solution of $(\calP_{TV})$. The construction heavily relies on specific properties of trigonometric polynomials.

Another related work is \cite{tang2013compressed}. This revolves around reconstructing signals $x \in \C^n$ of the form
\begin{align*}
	x_j = \sum_{k=1}^s c_k e^{i2\pi f_k j}, j \in \set{0, n-1},
\end{align*}
where $f_k$ are unknown frequencies in $[0,1]$. This can be seen as a semi-discretized analog of the super-resolution problem. They prove that if the signs of the $c_k$ are chosen independently according to the uniform distribution on the unit sphere in $\C$, $x$ can be reconstructed from $m$ random samples $x_j$, where $m \geqsim s\log(s) \log(n)$. Similar results, for more general measurement processes, are obtained in \cite{heckel2016generalized}: The authors of the latter paper are able to reduce the number of measurements to $\geqsim s\log(n)$. Note that the last factor $\log(n)$ in both expression makes it hard to directly generalize this to the infinite dimensional case. \newline

 In this section, we will use the machinery we have developed to derive a soft recovery statement for $(\calP_{TV})$. For this, we first have to embed the $(\delta_x)_{x \in \R}$ into a Hilbert space. We will choose the space $\calE$ defined in Example \ref{ex:E}. The inclusion of a filter in the definition of $\calE$ is in fact crucial to derive statement about closeness between peaks: In $\calE$ space, as we will see, $\delta_x$ and $\delta_{x'}$ are close if $x$ and $x'$ are close. This is in essence different to measuring their distance in $TV$-norm, where $\norm{\delta_x - \delta_{x'}}=2$ for arbitrary values of $x \neq x'$.
 
  Let us first note that $(\calP_{TV})$ is the atomic norm minimization with respect to the dictionary $(\delta_x)_{x \in\R}$ in $\calE$. This is actually a tiny bit less trivial to realize than one might first think:.
 
 \begin{lem}
 	The atomic norm with respect to the dictionary $(\delta_x)_{x\in \R}$ in $\calE$, where $\calE$ is defined as in Example \ref{ex:E}.3,  is equal to the $TV$-norm.
 \end{lem}
 \begin{proof}
  What is needed to be realized is that if $v= D\nu$ for a measure $\nu \in \calM(\R)$, then $v \in \calE$ is equal to $\nu$. To see this, let $w \in \calE$ be arbitrary. We then have
  \begin{align*}
  	\sprod{w, D\nu}_\calE = \int_\R \sprod{w,\delta_x}_\calE d\nu_x = \int_\R \int_\R \widehat{w}(t) \exp(-itx) \abs{\widehat{\phi}(t)}^2 dt d\nu(x) = \int_\R \widehat{w}(t) \overline{\widehat{\nu}(t)} \abs{\widehat{\phi}(t)}^2 dt = \sprod{w, \nu}_\calE,
  \end{align*}
  which shows that $D\nu=\nu$ as elements in $\calE$.
\end{proof}  
 As has already been mentioned, we will view the $\delta$-trains as signals in $\calE$ with sparse atomic decomposition in the dictionary $(\delta_x)_{x \in \calE}$. It will be convenient to define
  the \emph{auto-correlation function} $a$ through
 \begin{align*}
 	a = \overline{\phi}*\phi  = \calF^{-1}( \abs{\hatphi}^2),
\end{align*}  
where $\Phi$ is the filter defining $\calE$.

Let us begin by proving a small, but important lemma.
\begin{lem} \label{lem:FPhiF}
The map $\calF^{-1} \abs{\widehat{\phi}}^2 \calF : \calE \to \calC(\R)$ is well-defined and continuous, with norm smaller than $1$.
\end{lem} 
\begin{proof}
	If $v \in \calE$, then $v * \phi$ and therefore also $\widehat{v} \hatphi$ is an element of $L^2(\R)$, as is in particular $\overline{\hatphi}$. Thus, as a product of two $L^2$-functions, $\widehat{v} \abs{\hatphi}^2$ is  in $L^1$, and therefore its inverse Fourier transform is continuous.
	
	As for the norm bound, we have
	\begin{align*}
		\norm{ \calF^{-1}( \widehat{v} \abs{\hatphi}^2)}_\infty \leq \norm{\widehat{v} \abs{\hatphi}^2}_1 \leq \norm{\widehat{v}\hatphi}_2 \norm{\hatphi}_2 \leq \norm{v}_\calE
	\end{align*}
\end{proof}
 
 Let us now specialize the soft recovery condition to this setting.
 
 \begin{cor} \label{cor:g} Let $\calE$ be the space defined in Example \ref{ex:E}.3,  $\mu_0$ as in \eqref{eq:deltatrain} and  $i_0$, $a$ and $A$ as above. If there exists a $g \in \ran \calF^{-1} \abs{\widehat{\phi}}^2 \calF A^*$ with
 \begin{align}
 	\sum_{x \in I_0} \re(c_x^0 g(x)) \geq 1, \quad \abs{g({x_0})} \leq \sigma, \quad \sup_{x \in \R} \abs{ g(x) - a({x-x_0}) g(x_0)} \leq 1-t, \label{eq:gCond}
 \end{align}
 then for every solution $\mu^*$ of $(\calP_{TV})$, there exists an $x \in \supp \mu^*$ with
 \begin{align*}
 	\abs{a({x-x_0})} \geq \frac{\sigma}{t}.
 \end{align*}
 
 \end{cor}
 \begin{proof}
 	Since
 	\begin{align*}
 		\sprod{\nu, \delta_x}_\calE &= \sprod{\nu * \phi, \delta_x * \phi} = \int_{\R} \widehat{\nu}(t) \exp(-ixt) \abs{\widehat{\phi}}^2 dt = \calF^{-1} \left( \widehat{\nu}\abs{\hatphi}^2\right)(t), \\
 		\sprod{\delta_y, \delta_x}_\calE &= \int_{\R} \exp(iyt)\exp(-ixt) \abs{\widehat{\phi}}^2 dt = a_{x-y},
 	\end{align*}
  the claim follows immediately from Theorem \ref{th:soft}.
 \end{proof}

Let us compare this result with the exact recovery condition from e.g. \cite{candes2014towards}. In our notation, that condition asks for the function $g \in \ran A^*$ to obey
\begin{align}
	g(x) &= \tfrac{c_x^0}{\abs{c_x^0}}, x \in I_0, \label{eq:exactCond} \\
	  \text{ and } \sup_{x \in I} \abs{g(x)} &\leq 1. \nonumber
\end{align}
Hence, $g$ has to \emph{interpolate} the values $\tfrac{c_x^0}{\abs{c_x^0}}$ while still obeying $\sup_{x\in I} \abs{g(x)}\leq 1$ everywhere. Readers familiar with e.g. \cite{candes2014towards} know that this can be extremely tedious.

The condition described in Corollary \ref{cor:g} is not as rigid -- $g \in \ran \calF^{-1} \abs{\widehat{\phi}}^2 \calF A^*$ only have to \emph{approximate} \eqref{eq:exactCond}  and still obey $\sum_{x\in I_0} \re(c_x^0 g(x)) \geq 1$. This will typically make it necessary to have $\abs{g(x_0)}>1$ for $x\in I_0$, but this is possible by choosing $\sigma$ accordingly. The observant reader may have noticed that this will make $g$ automatically violate $\sup_{x \in \R} \abs{g(x)} \leq 1$. This is however not being asked by \ref{eq:gCond}! Figure \ref{fig:aSavesTheDay} makes it clear that it is reasonable to believe that $\sup_{x \in \R} \abs{g(x) - a(x-x_0) g(x_0)} \leq 1-t$ can be satisfied although $g(x_0) >1$. \newline

\begin{figure}
\centering
	\includegraphics[scale=.2]{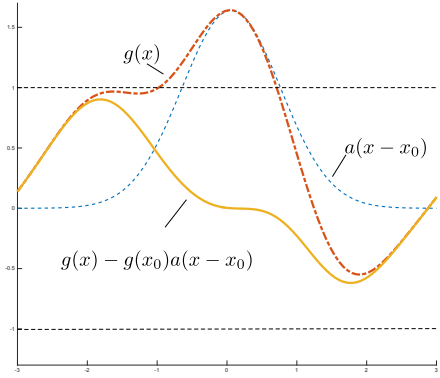}
	\caption{\label{fig:aSavesTheDay} The function $g$ violates $\sup_{x\in \R} \abs{g(x)} \leq 1$, but still obeys $\sup_{x\in\R} \abs{g(x) - g(x_0)a(x-x_0)} \leq 1-t$ for $t>0$.}
\end{figure}
 
In the following, we will describe a general strategy designing a measurement process $A : \calE \to \C^m$ for some $m \in \N$ such that the conditions \ref{eq:gCond} are satisfied. This measurement process of much more general nature than the measurement processes previously considered in the literature. The aim is hence not to propose a new measurement scheme, but rather to prove, on a conceptual level, the versatility of the soft recovery framework.
 
 Let $(f_n)_{n \in \N}$ be a \emph{Parseval frame} for $\calE$, meaning that 
 \begin{align*}
 	\sum_{n \in \N} \sprod{v, f_n} f_n = v
 \end{align*}
 for all $v \in E$. This means in particular that $\delta_{x_0} = \sum_{n \in \N} \sprod{\delta_{x_0}, f_n} f_n$. .
 
 Now let $M$ be a finite subset (say $\abs{M}=m$) of $\N$ and define the measurement operator $T_M$ as follows: 
 \begin{align*}
 	T_M : \calE \to \C^m, v \mapsto (\sprod{v,f_i})_{i \in M}, \quad \Rightarrow T^*_M : \C^m \to \calH , p \mapsto \sum_{i \in M} p_i f_i.
 \end{align*}
(This will be the map called $A$ in the main theorem).	In the following we will prove the existence of finite subsets $M \sse \N$ so that peaks $c_{x_0} \delta_{x_0}$ in $\mu_0$ can be  softly recovered from $T_M \mu_0$.
	
	The idea will be the following: If we could choose  $g$ in \ref{cor:g} as $(c_x^0)^{-1}\delta_x$, we would be done. In fact, already  the low-pass filtered version $\calF^{-1} \abs{\widehat{\phi}}^2 \calF \delta_x$ would suffice if we choose $\phi$, $\sigma$ and $t$ nicely. There is however no reason to believe that $\delta_x \in \ran T_M^*$.
	
	But since $\sum_{i \in \N} \sprod{\delta_{x},f_i} f_i = \delta_x$ for all $x$, we have $\ran T_M^* \ni T_M^*( ( \sprod{\delta_{x},f_i})_{i \in M} = \sum_{i \in M} \sprod{\delta_{x},f_i} f_i \approx \delta_x$, so we should be able to choose
	\begin{align*}
		g = \calF^{-1} \abs{\widehat{\phi}}^2 \calF \sum_{i \in M} \sprod{\delta_{x},f_i} f_i,
\end{align*}
and satisfy the conditions in \ref{cor:g} for good values of $\sigma$ and $t$. The crucial thing on which the success of this technique relies is exactly  how good the approximate equality $\sum_{i \in M} \sprod{\delta_{x},f_i} f_i \approx \delta_x$ is. Therefore, let us define
\begin{align}
	\epsilon(M, x_0):= \norm{ \sum_{i \in M} \sprod{\delta_{x_0}, f_i} f_i - \delta_{x_0}}_\calE
\end{align}

We will now prove a result about soft recovery properties of the program $(\calP_{TV})$. We will assume that our measure $\mu_0$  is of the form
\begin{align*}
	\mu_0 = \sum_{x \in I_0} c_{x}^0 \delta_{x},
\end{align*}
where $I_0$ is a subset of a (a priori known) compact interval $\calI$. The coefficients will for convenience assume to be normalized in the sense that $\norm{c^0}_1=1$. We will also assume a separability condition of the form
\begin{align*}
	\inf_{x, x' \in I_0} \abs{x-x'} \geq \Delta_{sep}.
\end{align*}

\begin{theo}
 \label{th:superResolution} Let $\lambda, L, \Delta, \Delta_{sep}>0$. Let further $a$ have the properties \begin{enumerate}
  \item $\abs{a(x)} \geq \lambda  \ \Rightarrow \ \abs{x} \leq \Delta$. \label{eq:Hump}
  \item $\abs{a(x)} \leq L$ for $\abs{x} \geq \Delta_{sep}$. \label{eq:decayA}
  \end{enumerate}
     Then for any compact interval $\calI$, there exists a measurement operator of the form $T_M$, $\abs{M} < \infty$, with the following property: If $\mu_0$ is as above and $x_0 \in I_0$ with $$   
  	\abs{c_{x_0}^0}(1+L)-L >\lambda
  	$$ then all solutions $\mu_*$ of $(\calP_{TV})$ obey
 \begin{align*}
 	 \exists x_* \in \supp \mu_* : \abs{x_*-x_0} \leq \Delta.
 \end{align*}
 If even $(\abs{c_{x_0}^0}(1+L)-L)(1-\gamma) \geq \lambda$ for some $\gamma >0$, we can furthermore secure that
 \begin{align*}
 	\abs{M} \leq \left\lceil\frac{2\abs{\calI}  \norm{\phi}_{1,2}}{\gamma} \right\rceil \cdot \sup_{x \in \calI} \abss{M\left(x, \frac{\gamma}{2\norm{\phi}_{1,2}}\right)}
 \end{align*}
 \end{theo}

The proof of this theorem can be found in Section \ref{sec:SuperResProof}. \newline

Let us finally look at a concrete example for a filter to get an idea of which tradeoff between $\lambda$, $\Delta$ and $\norm{\phi}_{1,2}$ can be achieved.

\begin{example}
	Consider the Gaussian filter
	\begin{align*}
		\phi(x) = \sqrt{\Lambda}{\pi}e^{-x^2/2(\Lambda^2)}.
	\end{align*}
	Then, $\phi$ is $\ell_2$-normalized. We have $\hatphi(t)= (\Lambda \pi)^{-1/2} \exp( -\Lambda^2 t^2/2)$ and thus
	\begin{align*}
		\norm{\phi}_{1,2}^2 = \frac{1}{\Lambda \pi}\int_\R t^2e^{-\Lambda^2 t^2}dt = \frac{1}{2\sqrt{\pi}\Lambda^2 },
	\end{align*}
	i.e. $\norm{\phi}_{1,2} \sim \Lambda^{-1}$. This also means that $\hatphi$ nowhere vanishes.
	
	The associated auto-correlation function can in fact be calculated explicitly: We have
	\begin{align*}
		a(x) = \frac{\Lambda}{\pi}\int_\R e^{-(t^2-(x-t)^2)/(2\Lambda^2)}dt = \exp(- x^2/)(2\Lambda^2))\frac{1}{\pi}\int_\R e^{-t^2-xt}dt = \exp(- x^2/(2\Lambda^2)).
	\end{align*}
	When applying the stronger version of Theorem \ref{th:superResolution}, we see that  we necessarily need to choose $\lambda \leq (1-\gamma)\abs{c_{x_0}^0}$. Therefore, let us set $\lambda = \theta (1-\gamma) \abs{c_{x_0}^0}$ for some $\theta \in [0,1]$. This corresponds to $\Delta \geq \Lambda \sqrt{2 \log( \theta (1-\gamma) \abs{c_{x_0}^0})^{-1}}$. Furthermore, this choice of $\lambda$ prohibits values of $\Delta_{sep}$ not obeying
	\begin{align*}
		\abs{c_{x_0}^0}(1+\exp(-(\Delta_{sep})^2/2\Lambda^2)) - \exp(-(\Delta_{sep})^2/2\Lambda^2) >\theta \abs{c_{x_0}^0} \\
		\Leftrightarrow \Delta_{sep} < \Lambda\sqrt{2 \log \left( \frac{ 1-\abs{c_{x_0}^0}} {\abs{c_{x_0}^0}(1-\theta)}\right)}.
	\end{align*}
	The dependence $\Delta/\Lambda$-- $\Delta_{sep}/\Lambda$ for a few values of $\abs{c_{x_0}^0}$ and $\gamma=\tfrac{1}{2}$ are depicted in Figure \ref{fig:Deltas}. We see that  if we assume a minimal separation $\Delta_{sep}=3\tfrac{19}{20}\Lambda$, peaks of relative size $\tfrac{1}{10}$ are guaranteed to deviate less than $\Delta= 2\tfrac{1}{10} \Lambda$ in the solutions to $(\calP_{TV})$. \newline

	\begin{figure}
	\centering
	\includegraphics[scale=.55]{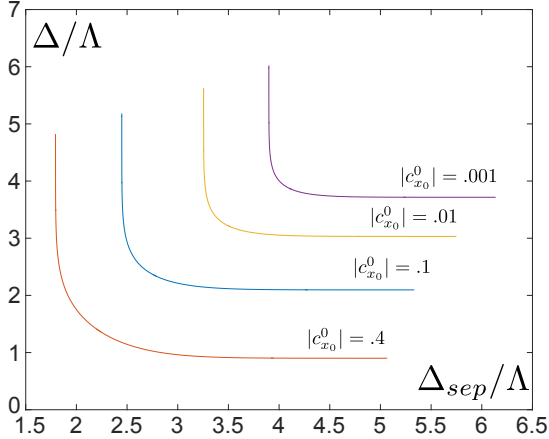}
		\caption{The relation $\Delta_{sep}/\Lambda$ -- $\Delta/\Lambda$ for a few values of $\abs{c_{x_0}^0}$. In the figure, $\gamma$ is fixed to be equal to $\tfrac{1}{2}$. \label{fig:Deltas}.}
	\end{figure}	
	
	As for the number of measurements (the size of the set $M$), we can say the following: $\phi$ is rapidly decaying, and thus in $B_{p,p}^\alpha$ for all $\alpha>0$, furthermore with $\norm{\phi}_{B_{p,p}}^s \asymp \Lambda^{-1}$. Using the frame we construct in Section \ref{sec:ParsevalFrame} of the Appendix (in particular invoking Corollary \ref{cor:FrameFullRate}), we have
	\begin{align*}
		\left\lceil\frac{2\abs{\calI}  \norm{\phi}_{1,2}}{\gamma} \right\rceil \cdot \sup_{x \in \calI} \abss{M\left(x, \frac{\gamma}{2\norm{\phi}_{1,2}}\right)} \leqsim \left\lceil \frac{\abs{\calI}  \Lambda^{-2+ \alpha^{-1}}}{\gamma^{2+\alpha^{-1}}}  \right\rceil,
		\end{align*}
		where the implicit constant only depends on $\alpha$. Hence, the Theorem \ref{th:superResolution} proves that the number of measurements needed to resolve a peak known to be lying in an interval $\calI$ will be proportional to length of that interval. Note also that the number of measurements grows as $\Lambda \to 0$. This is to be expected, since the quality of the soft recovery is enhanced for $\Lambda \to 0$.
\end{example}

\subsection*{Acknowledgement}

The author acknowledges support from the Deutsche Forschungsgemeinschaft (DFG) Grant KU 1446/18-1, and from the Berlin Mathematical School (BMS). 
He further wishes to thank Maximilian März, Martin Genzel and also Gitta Kutyniok for careful proof-reading and many important suggestions and remarks. These have significantly improved the quality of the paper. He also wishes to thank Ali Hashemi for valuable discussions regarding the numerical experiments in Section 4.3, as well as other topics.

\bibliographystyle{abbrv}
\bibliography{/homes/numerik/flinth/Documents/bibliographyCSandFriendsMASTER}

\section{Appendix} \label{sec:proofs}

Here, we provide some proofs left out in the main body of the text.

\subsection{The $\ell_{12}$-norm as an Atomic Norm} \label{sec:L12Proof} 
Here, we prove that the atomic norm with respect to the dictionary $\Phi^{12}$ is the $\ell_{12}$-norm.

	Let $D$ be the synthesis operator of $\Phi^{12}$. For a matrix $X \in \R^{k,d}$, let $\eta_i$ denote the normalized $i$:th column for $i$ with $X_i \neq 0$. Then for $V \in \R^{k,d}$ arbitrary
	\begin{align*}
		\sprod{V, D \left( \sum_{i: X_i \neq 0} \norm{X_i}_2 \delta_{i} \otimes \delta_{\eta_i} \right)}  = \sum_{i: X_i \neq 0} \int_{\sph^{k-1}} \sprod{V, \eta e_i^*} \norm{X_i}_2 d \delta_{\eta_i}(\eta) \delta_{k,i} = \sum_{ i \in [d]} \sprod{V_i, X_i} = \sprod{V,X}.
	\end{align*}
	I.e.  $X = D \left( \sum_{i: X_i \neq 0} \norm{X_i}_2 \delta_{i} \otimes \delta_{\eta_i} \right)$ and $\norm{X}_\calA \leq \normm{\sum_{i: X_i \neq 0} \norm{X_i}_2 \delta_{i} \otimes \delta_{\eta_i}}_{TV} = \norm{X}_{12}$. 
	
	On the other hand, we can write every measure $\mu$ on $[d] \times \sph$ as a sum $\sum_{i \in [d]} \mu_i \otimes \delta_{i}$ for some measures $\mu_i$ on $\sph$. In order for $D\mu=X$, we need
	\begin{align*}
		\sum_{i \in [d]} \int_{\sph^{k-1}} \sprod{V_i, \eta } d\mu_i(\eta) = \sprod{V,X}
	\end{align*}
	for $V \in \R^{k,d}$ arbitrary. This implies that $X_i = \int_{\sph^{k-1}} \eta d\mu_i$ for each $i$, from which we deduce $\norm{X_i}_2 \leq \int_{\sph^{k-1}} \norm{\eta}_2 d\abs{\mu}_i = \norm{\mu_i}_{TV}$, and hence
	\begin{align*}
		\norm{X}_{12} = \sum_{i \in [d]} \norm{X_i}_2  \leq \sum_{i \in [d]} \norm{\mu_i}_{TV} = \norm{\mu}_{TV}.
	\end{align*}
	Therefore, $\norm{X}_{\calA} \geq \norm{X}_{12}$, and the proof is finished.

\subsection{Proof of Proposition \ref{prop:sep}} \label{sec:ProofSep}
As was advertised, we will use the golfing scheme to prove Proposition \ref{prop:sep}. Let us make some preparations by arguing that it suffices to construct a $\nu$ in $\ran A^*$ with the following property: There exists an $\epsilon >0$ and a $\tau \in [\kappa (\hatsigma+ \epsilon), 1)$ such that
\begin{align}
	\abss{ \sprod{\nu, \psi_{i_0}} - \frac{\overline{c}_{i_0}^1}{\abs{c_{i_0}^1}} \widehat{\sigma}} &\leq \epsilon \label{eq:INollProx}  \\
	\widehat{\sigma} &=  \frac{ 1+ (1-\abs{c_{i_0}^1})(\tau - \kappa \hat{\sigma} ))}{\abs{c_{i_0}^1}}   \label{eq:HatSigma}\\
	\max_{\substack{\vphi \in \Phi\\ \vphi \neq \psi_{i_0}}} \abs{ \sprod{\nu, \vphi}} &\leq \tau - \kappa {\widehat{\sigma}}- \kappa \epsilon.  \label{eq:OffBound}	\\
	\tau &< 1- \kappa \hatsigma  \label{eq:tau}
\end{align}
To see this, we first convince ourselves that these conditions imply that $\nu$ satisfy \eqref{eq:Ankare}--\eqref{eq:orthCompSameSub} with $\sigma = \widehat{\sigma}+\epsilon$ and $t =1-\tau$. \eqref{eq:otherSubs} is in this case vacuous, and \eqref{eq:atPoint} is immediate from \ref{eq:INollProx}. To see that \eqref{eq:Ankare} is true, notice that \eqref{eq:INollProx} implies that $\re(\sprod{\nu, \psi_{i_0}c_i^0}) \geq \abs{c_{i_0}^0}(\widehat{\sigma} - \epsilon)$. This implies
\begin{align*}
	\re \left( \sum_{i=1}^n c_i^1 \sprod{\nu, \psi_i} + \sum_{i=1}^n c_i^2 \sprod{\nu, \theta_i}\right) &\geq \re( c_{i_0}^1 \sprod{\nu, \psi_{i_0}}) - \sum_{i \neq i_0} \abs{c_i^1 {\sprod{\nu, \psi_i}}}- \sum_{i=1}^n \abs{c_i^2 \sprod{\nu, \theta_i}} \\
	& \geq \abs{c_{i_0}^0}(\widehat{\sigma} - \epsilon) - (\norm{c^1}_1 - \abs{c_{i_0}^0} + \norm{c^2}_1)(\tau - \kappa \hatsigma) \\
	& \stackrel{\eqref{eq:HatSigma}}{=} 1 + (1-\abs{c_{i_0}^1}) (\tau - \kappa \hat{\sigma} ) - (1 - \abs{c_{i_0}^1})(\tau - \kappa \hatsigma) = 1.
\end{align*}
Finally, \eqref{eq:orthCompSameSub} holds since
\begin{align*}
	\sup_{ \vphi \in \Phi} \abs{\langle{\Pi_{\sprod{\psi_{i_0}}^\perp}\nu, \vphi} \rangle} = \sup_{ \vphi \in \Phi} \abs{\sprod{\nu, \vphi}- \sprod{\nu, \psi_{i_0}}\sprod{\psi_{i_0}, \vphi}} &\leq \begin{cases} 0 & \vphi= \psi_{i_0} \\ \abs{\sprod{\nu, \vphi}} + \kappa  \abs{\sprod{\nu, \psi_{i_0}}} & \vphi \neq \psi_{i_0} \end{cases} \\
	&\leq \tau - \kappa \hatsigma - \kappa \epsilon + \kappa \sigma = 1-t.
\end{align*}
Theorem \ref{th:soft} implies that there exists a $\vphi \in \Phi$ with 
\begin{align*}
	\abs{\sprod{\vphi, \psi_{i_0}}} \geq \frac{1-\tau}{\hatsigma} > \kappa,
\end{align*} and this together with the definition of $\kappa$ implies $\vphi = \psi_{i_0}$.

It now remains to construct a $\nu$ satisfying \eqref{eq:INollProx}--\eqref{eq:OffBound}. We can without loss of generality assume that $\Psi$ is the standard basis: By redefining $X \mapsto \Psi^*X$, $\Theta \mapsto \Psi^*\Theta$, we will consider the same minimization problems, not change the values of the parameters $\kappa$ and $M$, and still have isotropic measurements.

We now partition the $m$ rows of $A$ in to $r$ groups $\Lambda_i$ with $\Lambda_i = p$, define $A_i = \calP_{\Lambda_i}A$ and propose the following golfing iteration:
\begin{align*}
	\nu^0 &= 0 \\
	\nu^j &= \nu^{j-1} + \tfrac{1}{p} A_j^*A_j(\nu^{j-1}_{i_0} - \omega\hatsigma) e_{i_0}, \quad j = 1, \dots, r,
\end{align*}
where we introduced the notation $\omega := \tfrac{\overline{c}_{i_0}^1}{\abs{c_{i_0}^1}}$. The rest of the proof will consist of proving that $\nu:= \nu^r$ satisfies \eqref{eq:INollProx}--\eqref{eq:OffBound} with high probability, for a value of $\epsilon$ we will specify later. We will do this in three steps.  We will use the Bernstein inequality: If $Y_i$ are (complex) independent, centered random variables satisfying $\abs{Y_i} \leq R$ almost surely and $\sum_{i=1}^p \erw{\abs{Y_i}^2} \leq V$, then
\begin{align}
	\prb{\abss{\sum_{i=1}^r Y_i} \geq t} \leq 4 \exp\left(\frac{-t^2/4}{V + + Rt/\sqrt{2}}\right). \label{eq:Bernstein}
\end{align} 
Note the slightly worse constants compared to result most often referred to as the Bernstein inequality. This is due to the complexity of the variables: one may deduce the above result from the standard real result by observing
\begin{align*}
	\prb{\abs{Z} \geq t} \leq \prb{ \abs{\re{Z}} \geq \tfrac{t}{\sqrt{2}} \wedge \abs{\im{Z}} \geq \tfrac{t}{\sqrt{2}}}.
\end{align*} 

{\bf Step 1:} We have for each $j=0, \dots r$
\begin{align}
	\abs{  \nu^j_{i_0} - \omega \hatsigma} \leq 2^{-j} \hatsigma \label{eq:decay}.
\end{align}
We will prove this by induction. $j=0$  is trivial, so we only have to carry out the induction step. We have
\begin{align*}
	\abs{\nu^j_{i_0} - \omega \hatsigma} = \abs{1 - \tfrac{1}{p} (A_je_{i_0})^*A_je_{i_0}} \cdot \abs{\nu^{j-1}_{i_0} - \omega \hatsigma},
\end{align*}
so if we can prove $\abs{1 - \tfrac{1}{p} (A_je_{i_0})^*A_je_{i_0}} \leq \tfrac{1}{2}$ with high probability, we are done (note that due to the partitioning of the measurements, $A_j$ is independent of $\nu^{j-1}$, whence we can regard $\nu^{j-1}$ as a constant in the induction step.) We have
\begin{align*}
	 1 - \tfrac{1}{p} (A_je_{i_0})^*A_je_{i_0} =  \sum_{\ell \in \Lambda_j} \tfrac{1}{p}(1- \abs{X^\ell_{i_0}}^2).
\end{align*}
This is a sum of independent (real) random variables, which are centered due to the isotropy of $X$. Let us bound the $R$ and $V$ parameters in the Bernstein inequality: Due to the incoherence properties of $X$, we almost surely have $\tfrac{1}{p}\abs{1 - \abs{X^\ell_{i_0}}^2} \leq \tfrac{1}{p}(1+ M^2)$ and
\begin{align*}
	\sum_{\ell \in \Lambda_j}\erw{ \tfrac{1}{p^2}\abs{1 - \abs{X^\ell_{i_0}}^2}^2} = \tfrac{1}{p}\erw{1 - 2\abs{X^\ell_{i_0}} + \abs{X_{i_0}^\ell}^2} \leq \tfrac{1}{p}(1 + M^2\erw{ e_{i_0}^* (X^\ell)^*X^\ell e_{i_0}}) = \tfrac{1}{p}(1 + M^2).
\end{align*}
In the final step, we used the isotropy assumption. Bernstein now implies
\begin{align*}
	\prb{ \sum_{\ell \in \Lambda_j} \tfrac{1}{p}(1- \abs{X^\ell_{i_0}}^2) \geq \tfrac{1}{2}} \leq 4 \exp\left(\frac{-p/8}{(1+M^2)(1 + 1/(2\sqrt{2}))}\right),
\end{align*}
which is smaller than $\tfrac{\delta}{r}$ provided $m \geqsim (1+M^2) \log\left(\tfrac{r}{\delta}\right)$. A union bound over $j$ now proves that \eqref{eq:decay} is true with a failure probability smaller than $ \frac{\delta}{2}$. \newline

{\bf Step 2:} Denote $\tau^* = \tau - \kappa \hatsigma - \kappa \epsilon$. We have for $\vphi \in \Phi$ arbitrary
\begin{align*}
	\abs{\sprod{\nu^{r}, \vphi}} = \sum_{j \leq r} \tfrac{1}{p}\abs{\sprod{A_j^*A_j e_{i_0}, \vphi} } \abs{\nu_{i_0}^{j-1}- \omega \hatsigma} \leq \sum_{j \leq r} \tfrac{1}{p}\abs{\sprod{A_j^*A_j e_{i_0}, \vphi} } 2^{-j}\widehat{\sigma},
\end{align*}
where we applied \eqref{eq:INollProx}. Hence, if we prove
\begin{align*}
	\frac{1}{p}\abs{\sprod{A_j^*A_j e_{i_0}, \vphi} } = \frac{1}{p}\big \vert {\sum_{\ell \in \Lambda_j} \sprod{\vphi, X^\ell}\sprod{X^\ell,e_{i_0}}} \big \vert\leq \frac{\tau^*}{\hatsigma}
\end{align*}
for $j= 1, \dots, r$, we are done.

The random variable we need to control is a sum of independent variables. They are however not necessarily centered: Instead, we have
\begin{align*}
	\erw{ \sprod{\vphi, X^\ell}\sprod{X^\ell, e_{i_0}}} = \sprod{\vphi, e_{i_0}}.
\end{align*}
We will therefore devote our attention to the centered random variables $Y_\ell = \tfrac{1}{p}\left(\sprod{\vphi, X^\ell}\sprod{X^\ell, e_{i_0}} - \sprod{\vphi, e_{i_0}}\right)$ instead. They can all, again by incoherence, be bounded by $M^2 + \kappa$ almost surely.  As for $V$, we have
\begin{align*}
	\erw{\abs{Y_\ell}^2} &= \tfrac{1}{p^2}\left(\erw{ \abs{\sprod{\vphi, X^\ell}\sprod{X^\ell, e_{i_0}}}^2 - 2\re(\sprod{\vphi, X^\ell}\sprod{X^\ell, e_{i_0}}\sprod{\vphi, e_{i_0})} + \abs{\sprod{\vphi,e_{i_0}}}^2} \right)\\
	&=  \tfrac{1}{p^2}\erw{ \abs{\sprod{\vphi, X^\ell}\sprod{X^\ell, e_{i_0}}}^2} - \abs{\sprod{\vphi,e_{i_0}}}^2 \leq \tfrac{1}{p^2}\erw{ \abs{\sprod{\vphi, X^\ell}\sprod{X^\ell, e_{i_0}}}^2} \\
	&\leq \tfrac{1}{p^2} M^2 \erw{\abs{\sprod{X^\ell,e_{i_0}}}^2} = \tfrac{M^2}{p^2}.
\end{align*}
In the final step, we used the incoherence property as well as the isotropy assumption.

Now we can apply the Bernstein inequality to get
\begin{align*}
	\prb{ \abs{\sum_{i\in \Lambda_j} Y_i} \geq \frac{\tau^*- \kappa}{\hatsigma}} \leq 4 \exp\left(\frac{-p{\tau^*-\kappa}^2/(4\hatsigma^2)}{M^2 + (M^2+\kappa)(\tau^*-\kappa)/(\sqrt{2}\hatsigma)}\right).
\end{align*}
This is smaller than $\tfrac{\delta}{4nr}$ provided $p \geqsim \left(M^2\tfrac{\hatsigma^2}{{(\tau^*-\kappa)}^2} +(M^2 + \kappa)\tfrac{\hatsigma}{(\tau^*-\kappa)} \right) \log\left(\tfrac{nr}{\delta} \right)$. A union bound over $\phi \in \Phi$ and $j$ gives that \begin{align*}
{\sprod{\nu^r, \phi}} \leq {\frac{1}{p}\bigg\vert{\sum_{j=1}^r\sprod{A_j^*A_j e_{i_0}, \vphi} }}\bigg\vert + \abs{\sprod{\vphi, e_{i_0}}} \leq \tau^*- \kappa + \kappa = \tau^*.
\end{align*}
 with a failure probability larger than $\tfrac{\delta}{2}$.  \newline

{\bf Step 3:} Now we simply need to choose the parameters $r$, $\epsilon$ and $\tau$ wisely. First, let $\epsilon = \tfrac{s}{\abs{c_{i_0}^1}}$ with $s>0$ so that \begin{align*}
\frac{\gamma}{16}\leq\frac{s\kappa}{\abs{c_{i_0}^1}} &\leq \frac{\gamma}{12} \\
\Rightarrow\kappa\cdot \frac{4(1+s)}{\abs{c_{i_0}^1}} &\leq 1- \frac{2\gamma}{3}.
\end{align*} 
Next choose $\tau = 1-\kappa \hatsigma -\frac{\gamma}{3}$. \eqref{eq:HatSigma} then implies that
\begin{align*}
	\hatsigma = \frac{1}{1 + \kappa \frac{1-\abs{c_{i_0}^1}}{\abs{c_{i_0}^1}}} \left( \frac{1 +\tau (1-\abs{c_{i_0}^1}) +s}{\abs{c_{i_0}^1}}\right) \leq \frac{2 - \abs{c_{i_0}^1}+s}{\abs{c_{i_0}^1}}.
\end{align*}
We used that $\tau <1$. Consequently
\begin{align*}
	\tau^*- \kappa &= \tau - \kappa (\hatsigma +\epsilon+1)  =  1- \kappa( 2 \hatsigma + \epsilon +1) \\ &\geq  1- \frac{\gamma}{3} - \kappa \left(2\cdot\frac{2 - \abs{c_{i_0}^1}+s}{\abs{c_{i_0}^1}}  + \frac{s}{\abs{c_{i_0}^1}} +\frac{\abs{c_{i_0}^1}}{\abs{c_{i_0}^1}}\right) \geq 1- \frac{\gamma}{3} - \kappa\left(\frac{4(1+s)}{\abs{c_{i_0}^1}} \right) \geq \frac{\gamma}{3},
\end{align*}
and hence $$\frac{\hatsigma}{\tau^*-\kappa} \leq \frac{3}{\gamma}\left(\frac{4(1+s)}{\abs{c_{i_0}}}\right) \leq \frac{3}{\gamma}\left(1-\frac{2\gamma}{3}\right) \frac{1}{\kappa}= \frac{C_\gamma}{\kappa},$$
 so $p \geqsim \left(C_\gamma^2 \tfrac{M^2}{\kappa^2} + \tfrac{C_\gamma (M^2+\kappa)}{\kappa} \right) \log\left(\tfrac{nr}{\delta} \right)$ will imply what we need.

Finally, since 
\begin{align*}
	\frac{\hatsigma}{\epsilon} \leq \frac{\abs{c_{i_0}^1}}{s}\cdot \left(\frac{2+2s}{\abs{c_{i_0}^1}}-\frac{s}{\abs{c_{i_0}^1}}\right) \leq \frac{16\kappa}{\gamma} \left( \frac{1}{2}\left(1- \frac{2\gamma}{3}\right) \frac{1}{\kappa} - \frac{\gamma }{16 \kappa}\right) \leq \frac{8}{\gamma} -\frac{19}{3}
\end{align*}
choosing $r \geq \log\left(\frac{8}{\gamma}\right)$ will ensure $2^{-r}\hatsigma \leq \epsilon$, which together with property \eqref{eq:decay} proves \eqref{eq:INollProx}. The proof is finished.

\subsection{The Nuclear Norm as an Atomic norm} \label{sec:NucProof}
Here we rigorously prove that the nuclear norm is the atomic norm of the dictionary $(\vphi_{v \otimes u})_{ v\in \sph^{k-1}, u \in \sph^{n-1}}$. Let us  simplify the notation by writing $u \otimes v$ instead of $\vphi_{u \otimes v}$. Denote, as always, the synthesis operator of the dictionary by $D$.

Let $X \in \R^{k,n}$ be arbitrary and $\sum_{i=1}^k \sigma_i v_i \otimes u_i$ be its singular value decomposition. For $V \in \R^{k,n}$ arbitrary, we have
\begin{align*}
	\sprod{V,X} = \sum_{i=1}^k \sigma_i \sprod{V,v_i \otimes u_i} = \int_{\sph^{k-1} \times \sph^{n-1}} \sprod{V, v \otimes u} d \left(\sum_{i=1}^k \sigma_i \delta_{v_i \otimes u_i} \right)(v \otimes u).
\end{align*}
Hence, $X = D\left(\sum_{i=1}^k \sigma_i \delta_{v_i \otimes u_i} \right)$ and $\norm{X}_\calA \leq \sum \sigma_i = \norm{X}_*$.

To prove the opposite inequality, suppose that $X = D\mu$. Due to the definition of the singular value decomposition, we then have
\begin{align*}
	\sum_{i=1}^k \sigma_i = \sprod{\sum_{i=1}^k v_i \otimes u_i, X} = \int_{\sph^{k-1} \times \sph^{n-1}} \sum_{i=1}^k\sprod{v_i \otimes u_i, v \otimes u} d\mu(v,u) \leq \int_{\sph^{k-1} \times \sph^{n-1}} \Big \vert \sum_{i=1}^k\sprod{v_i,v} \sprod{u_i,u} \Big\vert d \abs{\mu}(v,u).
\end{align*}
Now since $(v_i)_{i=1}^k$ and $(u_i)_{i=1}^n$ are orthonormal systems in their respective spaces, Cauchy-Schwarz implies $$\Big \vert \sum_{i=1}^k\sprod{v_i,v} \sprod{u_i,u} \Big\vert \leq \left(\sum_{i=1}^k\sprod{v_i,v}^2\right)^{1/2} \left(\sum_{i=1}^k\sprod{u_i,u}\right)^{1/2} = \norm{u}_2\norm{v}_2 = 1,$$
and thus $\sum_{i=1}^k \sigma_i \leq \abs{\mu}(\sph^{k-1} \times \sph^{n-1}) = \norm{\mu}_{TV}$. The proof is finished.

\subsection{Numerical Estimation of the Statistical Dimension } \label{sec:numest}
	Let us begin by describing a very simple method with which one can numerically estimate the statistical dimension of a cone generated by convex constraints in general, and then discuss the specific details of our experiments.
	
	So say $C= \cone M = \bigcup_{\tau >0} \tau \cdot M$, where $M$ is convex. Let us first note that we then have \cite[Prop 3.1]{AmelunxLotzMcCoyTropp2014}
	\begin{align*}
		\delta(C) = d-\erw{C^{\circ}} = d - \erw{ \inf_{\tau>0} \dist(g, \tau M)^2}
	\end{align*}
	($C^{\circ}$ denotes the \emph{polar cone}.) Due to the fact that $g \to \dist(g, \tau M)$ is Lipschitz \cite[Section 5]{Phelps1958} and $g$ is Gaussian, $\dist(g, \tau M)$ will be highly concentrated around the expected value (this phenomenon is called \emph{measure concentration} -- see e.g. \cite{Ledoux2001}). Thus, by generating a few samples of $g$, calculating $\dist(g,\tau M)^2$ for each of them and finally taking the average, we will get a good approximation of $\delta(C)$.
	
	We still need to calculate $ \inf_{\tau>0}\dist(u, \tau M)^2$ for a fixed  $u$. For each fixed $\tau$, we may evaluate $J_u(\tau):=\dist(u, \tau M)^2$ by solving the convex optimization problem
	\begin{align}
		\min \norm{u - x}_2^2 \st x \in \tau M. \tag{$\calP_{\tau}$}
\end{align}	 
	The function $J_u$ is convex and thus has a unique minimum. If we furthermore can guarantee that $\norm{x}_2 \geq a$ for all $x \in M$, this minimum is attained on the compact interval $[0, \tfrac{\norm{u}_2}{a}]$ \cite[Lemma C.1]{AmelunxLotzMcCoyTropp2014}. We can hence find the minimum of $J_u$ by Golden Section Search (a version of ternary search specifically designed for the case that function evaluations are expensive) \cite{MR0055639}. \newline
	
	Now let us concretely look at the cones generated by \eqref{eq:SoftNuc1}--\eqref{eq:SoftNuc2} and \eqref{eq:ExactNuc}, respectively. Let us first note that due to the uniform distribution of $\ran A^*$, we may without loss of generality assume that the singular vectors $u_i$ and $v_i$ to be the canonical unit vectors. The version of $(\calP_\tau)$ corresponding to\eqref{eq:SoftNuc1}--\eqref{eq:SoftNuc2} then reads:
	\begin{align*}
		\min \norm{U -X}_F^2 \st && X_{11} &\leq  \tau\sigma \tag{$\calP_{\tau,\text{soft}}^{\sigma,t}$} \\
								& & \sum_{k=1}^r \sigma_i X_{ii} &\geq \tau \\
								&& \norm{X - X_{11} e_1 e_1^*}_{2\to 2} &\leq \tau(1-t).		
	\end{align*}
	The version corresponding to \eqref{eq:ExactNuc} is
	\begin{align*}
		\min \norm{U -X}_F^2 \st && X & =\begin{bmatrix}
		X^1 & 0 \\ 0 & X^2
\end{bmatrix} \tag{$\calP_{\tau,\text{exact}}$} \\
								& & X^1 &= \tau \id \in \R^{r,r} \\
								&&  \norm{X^2}_{2\to 2} &\leq \tau, X^2 \in \R^{k-r,n-r}.	
	\end{align*}
	In our numerical experiments, we for each value $\sigma =1\tfrac{1}{10}, 1\tfrac{2}{10}, \dots, 2$ and $t=\frac{1}{20}, \dots \tfrac{19}{20}$, and for $(\calP_{\tau, \text{exact}})$, generated $25$ samples of the Gaussian distribution using {\tt randn}. We used the MATLAB-package {\tt cvx} \cite{cvx} for solving $(\calP^{\sigma,t}_{\tau,\text{soft}})$ and $(\calP_{\tau, \text{exact}})$. Note that the constraints $(\calP^{\sigma,t}_{\tau,\text{soft}})$ can be infeasible for some values of $\sigma$ and $\tau$. If such a case was detected by \text{cvx}, we deduced that $C = \set{0}$, and $\delta(C)=0$.

\subsection{Construction of Parseval Frames for $\calE$ if $\hatphi$ is Non-Vanishing.} \label{sec:ParsevalFrame}
In this section, we will describe how one can transform frames for $L^2(\R)$ into frames of $\calE$. Let us begin by considering the map
\begin{align*}
	J : \calE \to L^2, v \mapsto v *\phi.
\end{align*}
$J$ is by definition of $\calE$ isometric, and therefore continuous and injective. If $\hatphi$ is non-vanishing, it is furthermore bijective. To prove see this, notice that for $w \in L^2$ arbitrary, $v=\calF(\hatphi^{-1} \widehat{w}) \in \calE$, and
\begin{align}
	\widehat{J(v)} = \hatphi^{-1} \widehat{w} \hatphi = \widehat{w},
\end{align}
which proves that $J(v) = \widehat{w}$. The bounded inverse theorem proves that $J$ is an isometry between Hilbert spaces, which in particular proves that $J^* J = \id$. \newline

Now let $(\alpha_i)_{i \in \N}$ be a Parseval Frame for $L^2(\R)$. Then $f_i = J^*\alpha_i$ is one of $\calE$, since for $v \in \calE$ arbitrary
\begin{align*}
	\sum_{i \in \N} \sprod{v, f_i} f_i = \sum_{i \in \N} \sprod{v, J^*\alpha_i} J^*\alpha_i = J^* \left( \sum_{i \in \N} \sprod{Jv, \alpha_i} \alpha_i\right) = J^*Jv = v.
\end{align*}
 We used the Parseval frame property of $(\alpha_i)_{i \in \N}$. \\
 
 Now let us have a look how the approximation rates of $(f_i)$ and $\alpha_i$ are related. We have
 \begin{align*}
 	\bigg \Vert { \sum_{i \in M} \sprod{\delta_x, f_i} f_i - \delta_x} \bigg\Vert_\calE = \bigg \Vert { J^*\left(\sum_{i \in M} \sprod{\delta_x, J^*\alpha_i} \alpha_i - J\delta_x\right)} \bigg\Vert_\calE = \bigg \Vert { \sum_{i \in M} \sprod{\delta_x, \alpha_i} \alpha_i - J\delta_x} \bigg\Vert_2,
 \end{align*}
 meaning that \emph{the approximation rate of $(f_i)$ for $\delta_x$  in $\calE$ is the same as the approximation rates of $(\alpha_i)$ for $J\delta_x = \phi(\cdot -x)$}. If $\phi$ is smooth, we can choose $(\alpha_i)$ to be a Parseval frame consisting of \emph{wavelets} to get a good rate, as the following results shows.
 
\begin{theo} (DeVore \cite[p.117]{devore1998nonlinear}, simplified version.) Let $s>0$, $\frac{1}{p}= \tfrac{1}{2}+s$ and suppose that $\psi$ is a function obeying:
\begin{itemize}
\item $\psi$ has $r$ vanishing moments, with $r>s$.
\item $\psi \in B_{q,p}^\rho(\R)$ for some $q$ and $\rho>s$.
\end{itemize}
Then there for every $n \in \N$ exists a set $K$ with $\abs{K} \leq N$ such that there exists coefficients $(c_{i})_i \in K$ with $\abs{K} \leq n$n
\begin{align*}
	\norm{\sum_{i \in K } c_{i} \psi_i  -f}_2 \leq C \norm{f}_{B_{p,p}^s} n^{-s},
\end{align*} 
where $(\psi_{i})$ is the  system of wavelets generated by $\psi$. $C$ is only dependent on $s$.
	
\end{theo}
 
 We conclude that there for every smooth filter $\phi$ exists a Parseval frame $(f_i)$ for $\calE$ with
 \begin{align} \label{eq:AppRate}
 	\sup_{x \in \R} \abs{M(x, \epsilon)} \leqsim \norm{\phi}_{B_{p,p}^s} \epsilon^{-1/s},
 \end{align}
 where the implicit constant only depends on $s$.
 
Although we know that there for each $x \in \R$ exists a (minimal) set $M(x, \epsilon_0)$ with $\epsilon(x,M(x, \epsilon_0))$ smaller than some fixed $\epsilon$, say, there is no reason why there should exist an $M$ such that $\epsilon(M,x_0)< \epsilon$ for all $x_0 \in \R$. If we however a priori know that $x_0$ lies in some compact interval $\calI$, we can do a bit better: then, there exists a set $M$ such that $\sup_{x \in \calI} \epsilon(x,M)\leq \epsilon$.

\begin{lem} \label{lem:MLemma}  Let $\calI$ be an interval of finite length, and  $\epsilon >0$. Then there exists an $M_\calI$ with $\sup_{x \in \calI} \epsilon(x,M)\leq \epsilon$. $M_\calI$ obeys
\begin{align*}
	\abs{M_\calI} \leq \left\lceil\frac{2\abs{\calI}  \norm{\phi}_{1,2}}{\epsilon} \right\rceil \cdot \sup_{x \in \calI} \abss{M\left(x, \frac{\epsilon}{2\norm{\phi}_{1,2}}\right)}< \infty.
\end{align*}
\end{lem}
\begin{proof}
	Let us first note that for $M$, $x$ and $y$ arbitrary, we have
	\begin{align*}
		\normm{ \left(\sum_{i \in M} \sprod{\delta_{x}, f_i} f_i - \delta_{x}\right)-\left(\sum_{i \in M} \sprod{\delta_{y}, f_i} f_i - \delta_{y} \right) }^2_\calE &\leq \norm{ \id -\sum_{i \in M} f_i \otimes f_i}^2 \norm{\delta_x - \delta_y}_\calE^2 \\
		&\leq \int_{\R} \abs{\exp(ixt) - \exp(iyt)}^2 \abs{\hatphi(t)}^2 dt \\ &\leq \int_{\R} \abs{ixt - iyt}^2 \abs{\hatphi(t)}^2 dt= \abs{x-y}^2 \norm{\hatphi}_{1,2}^2.
	\end{align*}
		We used that $\norm{A - B} \leq \norm{A}$ for $A$, $B$ positive semi-definite. Now let $C=(x_i)_i$ be a collection of points $x_i$ so that $$\calI \sse \bigcup_i U_{r}(x_i),$$ with $r= \frac{\epsilon}{2\norm{\phi}_{1,2}}$ It is clear that we can choose this collection in such a way so that $\abs{C} \leq \left\lceil\frac{2\abs{\calI}  \norm{\phi}_{1,2}}{\epsilon} \right\rceil$. Now define $M$ as
	\begin{align*}
		M = \bigcup_{x \in C}M\left(x, \tfrac{\epsilon}{\norm{\phi}_{1,2}} \right).
	\end{align*}
	For $y \in \calI$ arbitrary, let $x_i$ be the point in $C$ closest to $y$. Then
	\begin{align*}
		\Big\Vert \sum_{j \in M} \sprod{\delta_{y}, f_j} f_j - \delta_{y} \Big \Vert_\calE \leq \Big\Vert \sum_{j \in M} \sprod{\delta_{x_i}, f_j} f_j - \delta_{x_i} \Big \Vert_\calE  + \norm{\phi}_{1,2} \abs{x_i-y} \leq \epsilon_0  \left( \frac{1}{2} + \frac{1}{2\norm{\phi}_{1,2}}\right) \leq \epsilon_0
	\end{align*}
		We used that $\norm{\phi}_{1,2} \geq \norm{\phi}_2 =1$. We now only need to prove that $\sup_{x \in \calI} \abs{M(x, \epsilon)}< \infty$ for each $\epsilon$. This is however easy: $\calI$ is compact and the function $x \to  \abs{M(x, \epsilon)}$ is subcontinuous - if $M_x$ is a set for which $\epsilon(x, M_x) < \epsilon$, $\epsilon(y, M_x) < \epsilon$ also for $y$ in a neighborhood of $x$. This prove that $\abs{M(y, \epsilon)} \leq \abs{M_x}$
in that neighborhood.\end{proof}

Combining Equation \eqref{eq:AppRate} with Lemma \ref{lem:MLemma}, we finally obtain:

\begin{cor} \label{cor:FrameFullRate}
	Suppose that $\phi \in B_{p,p}^s(R) \cap W^{1,2}(\R)$, with $p= 2(1+2s)^{-1}$, and $\epsilon >0$. Then there exists a Parseval Frame $(f_i)$ for $\calE$ with the following property: For each compact interval $\calI$, there exists an $M_\calI$ with $\sup_{x \in I} \epsilon(x, M) \leq \epsilon$ obeying
	\begin{align*}
		\abs{M_\calI} \leqsim  \norm{\phi}_{B_{p,p}^s}\left\lceil\frac{2\abs{\calI}  \norm{\phi}_{1,2}}{\epsilon} \right\rceil \cdot \left\lceil \frac{\epsilon}{2\norm{\phi}_{1,2}}\right\rceil^{-1/s}
	\end{align*}
\end{cor}

 \subsection{Proof of Theorem \ref{th:superResolution}} \label{sec:SuperResProof}

Let $\epsilon_0 >0$ be so small so that 
 	$$(1-\epsilon_0)( \abs{c_{x_0}^0} ( 1+L) - L )\geq \lambda$$ and invoke Lemma \ref{lem:MLemma} to secure the existence of an $M$ with $\sup_{x \in \cal I} \epsilon(M, x) \leq \epsilon_0$. If $(1-\gamma) \abs{c_{x_0}^0} ( 1+L) - L )\geq \lambda$, we can choose $\gamma = \epsilon_0$, and even get a set with the desired cardinality bound. Let us now define
 	\begin{align*}
 		g = C\sum_{i \in M} \sprod{\omega_{x_0} \delta_{x_0}, f_i}_\calE \calF^{-1} (\abs{\hatphi}^2 \widehat{f}_i),
 	\end{align*}
 	where $C=(\abs{c_{x_0}^0}(1 - \epsilon_0 + (1-\epsilon_0) L) - (1+\epsilon_0)L)^{-1}$ and $\omega_{x_0} = \abs{c^0_{x_0}}^{-1} \overline{c}^0_{x_0}$. $g$ is then in $\ran \calF^{-1} \abs{\hatphi}^2 \calF T_M^*$. It furthermore obeys for every $y \in \R$
 	\begin{align*}
 		\abs{g(y) - C\omega_{x_0} a(y-x_0)} &=  C\abss{\calF^{-1}\left(\abs{\hatphi}^2 \calF\left(\sum_{i \in M} \sprod{ \delta_{x_0}, f_i}f_i -\delta_{x_0}\right)\right)(y) } \\
 		&\leq C \Big\Vert {\sum_{i \in M} \sprod{ \delta_{x_0}, f_i}f_i -\delta_{x_0}} \Big\Vert_\calE 
 		\leq C \cdot \sup_{x \in \calI} \epsilon(M, x) \leq C\epsilon_0 .
 	\end{align*}
 	We used \ref{lem:FPhiF}. Consequently,
 	\begin{align*}
 		\abs{g(x_0)} &\leq  C \abss{ \omega_{x_0} a(x_0-x_0)} + Cs\epsilon_0 =C(1+ s\epsilon_0) \\
 		\abs{g(y) - g(x_0) a(y-x_0)} &\leq \abss{g(y) - C\omega_{x_0} a(y-x_0)} + \abss{C\omega_{x_0} -g(x_0)} \abs{a(y-x_0} \leq 2 C \epsilon_0.
 	\end{align*}
 	It remains to bound $\sum_{x \in I_0} \re(c_x^0 g(x))$. We have
\begin{align*}
	 \sum_{x \in I_0} \re(c_x^0 g(x))  &\geq C\sum_{x \in I_0}   \re \left(c_x^0 \omega_{x_0}a(x-x_0)\right)- \epsilon_0 \abs{c_x^0 a(x-x_0)}\\
	 &= C\abs{c_{x_0}^0}(1-\epsilon_0)  - (1+\epsilon_0)C \sum_{x \in I_0 \backslash \set{x_0}} \abs{ c_x^0 a(x-x')} \\
	 &\geq C \left(\abs{c_{x_0}^0} (1- \epsilon_0) - (1+\epsilon_0) L (1- \abs{c_{x_0}^0})\right) \\
	 &= C \left( \abs{c_{x_0}^0}(1 - \epsilon_0 + (1-\epsilon_0) L) - (1+\epsilon_0)L \right)=1.
	 \end{align*}All in all, we have proven that $g$ obeys \eqref{eq:gCond} for $\sigma = C(1+s\epsilon_0)$ and $t = 1- 2C\epsilon_0$. Corollary \ref{cor:g} implies that there exists an $x^*$ in $\supp \mu_*$ for which
	 \begin{align*}
	 	\abs{ a(x_* -x_0)} \geq \frac{1-2C\epsilon_0}{C(1+s\epsilon_0)} \geq \frac{1}{C}\left(1- (2C+1)\epsilon_0\right) \geq \lambda
	 \end{align*}
which by property \ref{eq:Hump} of $a$ implies the statement of the theorem.
 
\end{document}